\documentclass{article}

\usepackage{amsmath}
\usepackage{amssymb}
\usepackage{amsthm}
\usepackage{comment}
\usepackage{mathtools}
\usepackage[title,toc]{appendix}
\usepackage{hyperref}
\usepackage{bbm}
\usepackage[utf8]{inputenc}
\usepackage{tikz-cd}
\usepackage{wasysym}
\usepackage{varwidth}
\usepackage{mathtools}
\usepackage{xspace}
\usepackage{aliascnt}

\makeatletter
\renewcommand*\sectionautorefname{\S\@gobble}
\renewcommand*\subsectionautorefname{\S\@gobble}
\renewcommand*\subsubsectionautorefname{\S\@gobble}
\makeatother

\mathchardef\mhyphen="2D
\DeclareMathOperator{\Hom}{Hom}

\DeclareMathOperator{\pt}{pt}
\DeclareMathOperator{\FinSet}{\mathbf{FinSet}}

\DeclareMathOperator{\Cat}{\mathbf{Cat}}

\DeclareMathOperator{\End}{End}
\DeclareMathOperator{\Rep}{Rep}
\newcommand{\Vect}{\mathbf{Vect}}

\DeclareMathOperator{\Id}{Id}

\DeclareMathOperator{\Sh}{Sh}

\DeclareMathOperator{\Irr}{Irr}
\DeclareMathOperator{\gr}{\mathbf{gr}}

\newcommand{\GL}{GL}
\newcommand{\CCC}{\mathcal{C}}

\newcommand{\CatAd}{\mathbf{2}\mhyphen\Vect_{\gr}}
\newcommand{\CatAdTen}{{\mathbf{2}\mhyphen\Vect_{\gr}}^{\otimes}}
\newcommand{\ZAd}{{\Zmod_{\gr}}}
\newcommand{\ZAdTen}{{\Zmod_{\gr}}^\otimes}
\newcommand{\base}[1]{#1_{base}}
\newcommand{\nobase}[1]{#1^{\circ}}

\newcommand{\FinDisj}{\FinSet^{\sqcup}}
\newcommand{\GSet}{G\operatorname{Set}}

\DeclareMathOperator{\PolyFun}{PolyFun}

\newcommand{\Hop}{\mathcal{H}}
\newcommand{\Heis}{\operatorname{Heis}}

\newcommand{\FF}{\mathbbm{F}}
\newcommand{\NN}{\mathbbm{N}}
\newcommand{\ZZ}{\mathbbm{Z}}
\DeclareMathOperator{\point}{\pt}
\newcommand{\kk}{\mathbbm{k}}
\newcommand{\CC}{\mathbbm{C}}
\newcommand{\PPS}[1]{{\mathcal{P}^{\otimes #1}}}
\newcommand{\CS}[1]{{\mathcal{C}^{\otimes #1}}}
\newcommand{\AAS}[1]{A^{\otimes #1}}
\newcommand{\PP}{\mathcal{P}}

\newcommand{\cartesianarrow}{\xrightarrow{+}}

\newcommand{\DDD}{\mathcal{D}}

\newcommand{\EndAd}{\End_{ad}}

\newcommand{\Zmod}{\ZZ\mhyphen\operatorname{Mod}}

\newcommand{\fset}[1]{[#1]}
\newcommand{\pset}[1]{[#1]_*}
\newcommand{\comma}{adCartesian\xspace}

\DeclareMathOperator{\Sym}{Sym}

\newcommand{\longequal}{=\joinrel=\joinrel=}

\newcommand{\tik}{\begin{tikzcd}}
\newcommand{\tak}{\end{tikzcd}}

\makeatletter
\def\latearrow#1#2#3#4{%
  \toks@\expandafter{\tikzcd@savedpaths\path[/tikz/commutative diagrams/every arrow,#1]}%
  \global\edef\tikzcd@savedpaths{%
    \the\toks@%
    (\tikzmatrixname-#2)
    to%
    node[/tikz/commutative diagrams/every label] {$#4$}
    (\tikzmatrixname-#3)
;}}
\makeatother

\def\stik#1#2{
\begin{tikzpicture}[baseline= (a).base]%
\node[scale=#1] (a) at (0,0){%
\begin{tikzcd}[ampersand replacement=\&]%
#2
\end{tikzcd}%
};%
\end{tikzpicture}
}

\def\widestik#1#2{
\begin{tikzpicture}[baseline= (a).base]%
\node[scale=#1] (a) at (0,0){%
\begin{tikzcd}[ampersand replacement=\&,column sep=large]%
#2
\end{tikzcd}%
};%
\end{tikzpicture}
}

\def\nstik#1#2{
\begin{tikzpicture}[baseline= (a).base]%
\node[scale=#1] (a) at (0,0){%
\begin{tikzcd}[ampersand replacement=\&,row sep=small,column sep=large]%
#2
\end{tikzcd}%
};%
\end{tikzpicture}
}

\def\nnstik#1#2{
\begin{tikzpicture}[baseline= (a).base]%
\node[scale=#1] (a) at (0,0){%
\begin{tikzcd}[ampersand replacement=\&,row sep=tiny]%
#2
\end{tikzcd}%
};%
\end{tikzpicture}
}

\def\smallttik#1{
\begin{tikzpicture}[anchor=base, baseline,inner sep=0]%
\node[scale=0.6] (a) at (0,0){%
\begin{tikzcd}[ampersand replacement=\&]%
#1
\end{tikzcd}%
};%
\end{tikzpicture}
}

\newcommand{\mapstack}[2]{
\begin{tikzpicture}[anchor=base, baseline,inner sep=0, row sep=0]%
\node[scale=0.6] (b) at (0,0.3){
$#1$
};%
\node[scale=0.6] (a) at (0,0){%
$#2$
};%
\end{tikzpicture}
}

\newtheorem{Theorem}{Theorem}

\newtheorem{Proposition}{Proposition}[section]

\newaliascnt{Corollary}{Proposition}
\newtheorem{Corollary}[Corollary]{Corollary}
\aliascntresetthe{Corollary}

\newaliascnt{Lemma}{Proposition}
\newtheorem{Lemma}[Lemma]{Lemma}
\aliascntresetthe{Lemma}

\newtheorem* {Claim}{Claim}

\newtheorem{Conjecture}{Conjecture}

\theoremstyle{definition}
\newaliascnt{Definition}{Proposition}
\newtheorem{Definition}[Definition]{Definition}
\aliascntresetthe{Definition}

\theoremstyle{remark}
\newaliascnt{Notation}{Proposition}
\newtheorem{Notation}[Notation]{Notation}
\aliascntresetthe{Notation}

\newaliascnt{Remark}{Proposition}
\newtheorem{Remark}[Remark]{Remark}
\aliascntresetthe{Remark}

\newtheorem* {Note}{Note}

\newaliascnt{Example}{Proposition}
\newtheorem{Example}[Example]{Example}
\aliascntresetthe{Example}

\newtheorem* {Observation}{Observation}

\begin{document}
\begin{abstract}
   Motivated by the work of of A. Zelevinsky on positive self-adjoint Hopf algebras, we define what we call a \emph{symmetric self-adjoint Hopf} structure for a certain kind of semisimple abelian categories. It is known that every positive self-adjoint Hopf algebra admits a natural action of the associated Heisenberg double. We construct canonical  morphisms lifting the relations that define this action on the algebra level and define an object that we call \emph{a categorical Heisenberg double} that is a natural setting for considering these morphisms. As examples, we exhibit the symmetric self-adjoint Hopf structure on the categories of polynomial functors and equivariant polynomial functors. In the case of the category of polynomial functors we obtain categorification of the Fock space representation of the infinite-dimensional Heisenberg algebra. 
   \end{abstract}

\title{Symmetric self-adjoint Hopf categories and a categorical Heisenberg double}
\author{Adam Gal, Elena Gal}
\date{\today}
\maketitle

\numberwithin{equation}{section}
\tableofcontents
\section{Introduction}
In this article we study a categorical analog of a certain class of Hopf algebras: namely Hopf algebras endowed with an inner product with respect to which the maps of multiplication and comultiplication are adjoint. Examples of such Hopf algebras arise naturally in many contexts. A principal example which we consider throughout the article is the Hopf algebra $\bigoplus_n K(\Rep(S_n))$ with multiplication and comultiplication given by induction and restriction functors and inner product given by the dimensions of the $\Hom$-spaces.  An abstract theory of this kind of  algebras was developed by Zelevinsky in the book \cite{Zelbook} under the name of positive self-adjoint Hopf (PSH) algebras. Certain representation-theoretic information, e.g. decomposition into irreducibles, can be derived from the abstract theory of PSH algebras.

All natural examples of PSH algebras have a categorical origin - in fact, they all appear as Grothendieck groups of categories of representations of sequences of groups or algebras. As such, it is logical to expect that the notion of PSH algebra has a categorical counterpart. In this article we propose a definition of such an object, which we call a symmetric self-adjoint Hopf (SSH) category. We apply this definition to show that any such category is equipped a canonical categorical action.   

To construct the categorification of the notion of PSH algebra we observe that the property of self-adjointness allows to essentially only consider the structure of multiplication in the algebra or the category. The Hopf axiom in this approach can be viewed as a property of the multiplication. 

On the categorical level this suggests to define a symmetric self-adjoint Hopf (SSH) category structure as a functor $\Hop$ from the category of finite sets $\FinSet$  to $\CatAd$ - the category of graded 2-vector spaces with 1-morphisms that admit adjoints. Such a functor defines the categorical analog of multiplication. $\Hop$ shouldn't preserve composition on the nose, so we have to work in a framework naturally allowing for these kind of functors. For the examples treated in this article it is practicable to define $\Hop$ as a functor of bicategories. To treat the Hopf axiom as a property of multiplication we require that Cartesian squares in $\FinSet$ should go to squares satisfying the Beck-Chevalley condition (\autoref{sec:BC}) in $\CatAd$. This allows us to associate an isomorphism to any image of a Cartesian square under $\Hop$. 
In particular, the categorical Hopf axiom is given by the image of a certain Cartesian square of finite sets under $\Hop$. This allows for a concise definition of self-adjoint Hopf category without resorting to the use of complicated coherence relations that appear in the literature on the subject. 
The notion of Hopf structure for semisimple abelian categories was first considered by Crane and Frenkel in  \cite{FrenkelCrane} and the coherence diagrams from their definition follow immediately in our framework (see \autoref{frenkelcranehopfcat}). 
Moreover, we show that in $\CatAd$ a square satisfies the Beck-Chevalley condition if and only if it is a certain $2$-categorical generalization of a Cartesian square (an \emph{\comma} square). In this sense the Hopf axiom can be thought as a continuity condition on $\Hop$, i.e. the requirement that $\Hop$ preserves \comma squares. The relevant definitions, as well as some properties of SSH categories are contained in \autoref{sec:SSHDef}.

Denote the finite set $\{1,\ldots,n\}$ by $\fset{n}$. We will say that $\Hop$ gives a structure of SSH category on the category $\Hop(\fset{1})$ and that $\Hop(\fset{1})$ is an SSH category. We show in \autoref{ssec:SSHbasicproperties} that the Grothendieck group of an SSH category has a canonical structure of PSH algebra. In this sense the notion of SSH category is a categorification of the notion of PSH algebra. The category $\Hop(\fset{1})$ has a canonically defined symmetric monoidal structure (hence the word \emph{symmetric} in the designation).

Two of the principal examples of PSH algebras considered in the book \cite{Zelbook} are the algebras $\bigoplus_n K(\Rep(S_n))$ and $\bigoplus_n K\left(\Rep(S_n[G])\right)$ where $S_n[G]$ is the wreath product $G^n\rtimes S_n$. The parallel examples of SSH categories that we consider in this work are the category $\PP$ of polynomial functors defined by Friedlander and Suslin in \cite{FrieSusPoly} and the category $\PP_G$ of equivariant polynomial functors defined in \autoref{ssec:defpequivariant}. In characteristic 0 both of these categories are semisimple and the SSH structure on them descends to the PSH structure on $\bigoplus_n K(\Rep(S_n))\cong K(\PP)$ and $\bigoplus_n K\left( \Rep(S_n[G])\right)\cong K(\PP_G)$ respectively. These examples are treated in \autoref{sec:p} and \autoref{sec:pequivariant}.

In \autoref{prop:heisdouble} we show (following Zelevinsky) that any positive self-adjoint Hopf algebra $A$ admits a canonically defined action of the Heisenberg double $(A,A)$. This action is given by multiplication by the elements of $A$ and their adjoints which satisfy certain relations. The main result of this work is the proof in \autoref{sec:heisenbergaction} that every SSH category admits a categorical analog of such an action. To elaborate, a certain reformulation of the relations between multiplication and comultiplication giving the Heisenberg double action in $\ref{heisenbergrelation}$ allows us to interpret the categorical action as a statement about the existence of a certain 2-isomorphism in Theorem \ref{th:deltam}. This 2-isomorphism arises as a consequence of certain squares - which are part of the SSH structure - being Beck-Chevalley squares. 

A remarkable feature of this isomorphism is that it arises on the category level from a canonical square with a non-invertible 2-morphism satisfying the Beck-Chevalley condition. Such a square is a purely categorical construct which is not detectable on the vector space level, where it descends to a non-commutative diagram. In other words the Heisenberg double action on the vector space level can only be formulated in terms of relations between multiplication and co-multiplication, whereas on the category level multiplication is enough and hence we are not required to choose an adjoint representing the co-multiplication.

To describe this approach in a more conceptual way we introduce a new category $\Heis{\Hop}$ that we call a \emph{categorical Heisenberg double} associated to the functor $\Hop$ equipped with a canonical projection to $\CatAd$. We use certain conjectural properties of this category stated in \S\ref{heuristics} to compute this square as an $\comma$ square in $\CatAd$. We conjecture that the category $\Heis{\Hop}$ is a categorification of the algebraic notion of Heisenberg double in the sense that the $2$-morphism structure of the action is given by the squares and relations between them in $\Heis(\Hop)$. It is our hope that $\Heis{\Hop}$ will be useful for constructing new canonical categorifications given the generalization of the notion of Hopf category to various frameworks such as dg- or stable $\infty$-categories.

The simplest and the most familiar example of the Heisenberg double construction is the Heisenberg algebra of infinite rank. It is the Heisenberg double associated to the PSH algebra $\Lambda$ of symmetric functions. Note that $\Lambda \cong \bigoplus_n K(\Rep(S_n))\cong K(\PP)$. The Heisenberg algebra is traditionally defined in the literature in terms of an infinite number of generators and relations; there are several common presentations, some of which, as well as their relation to the generator-independent definition used in this article we outline in \autoref{ssec:generators}. 
Khovanov proposed, in \cite{Khov}, an approach to categorification of the Fock space action of this algebra in a diagrammatic language of generators and relations for endofunctors of a category and $2$-morphisms between them. In \cite{hongyacobiPoly} Hong and Yacobi describe this action on the category $\PP$ of polynomial functors. Since $\PP$ is our basic example of an SSH category, it is natural to compare our construction of categorical Heisenberg double action to the former. In \autoref{ssec:FockP} we explicitly construct the morphism lifting the Heisenberg relation in this case.

In fact, the original motivation for this article was our desire to understand which parts of Khovanov's construction (and it's realisation on $\PP$ by \cite{hongyacobiPoly}) can be derived from the categorical version of a PSH structure via categorical constructions. The 2-isomorphisms we construct in Theorem \ref{th:deltam} are part of Khovanov's 2-morphism structure - these 2-morphisms are exactly what is needed to construct an embedding of the Heisenberg algebra into the Grothendieck group of Khovanov's Heisenberg category. Moreover, it is possible to recover all of the graphical calculus constructed in \cite{Khov} from an SSH structure on a category. Given a category with a SSH structure and a categorical analog of a primitive element in the algebra one can construct the Heisenberg action in the sense of Khovanov by choosing adjunction data satisfying certain compatibility conditions. This suggests a possible approach to proving Khovanov's conjecture about that the $K$-group of his categorification is exactly the Heisenberg algebra: we should show that the Heisenberg double construction commutes with taking $K$-group. Furthermore, the input category in the Heisenberg double construction can be changed quite flexibly and we plan in a future paper to consider this construction in the positive characteristic setting - as suggested by R. Rouquier.

In the article \cite{savageyacobi} by Savage and Yacobi the authors propose to consider the isomorphisms \eqref{list:categoricalFock} (which they call "categorified Mackey relations") for categorifying the Fock space representation of the Heisenberg doubles given by the $K$-groups of categories of modules over towers of algebras and construct these for a certain class of examples. It would be interesting to attempt to formulate their results in the generalization of the SSH categories framework described here to the non-semisimple category case, and this might provide additional insight into the structure inherent in these examples (see also \autoref{rem:leftright}). 

The Heisenberg algebra action and its categorification also appears naturally in several places (\cite{Grojnowski}, \cite{Nakajima} and the categorical version in \cite{CautisLicataHilbert}, also \cite{CautisLicataQuantum} to cite a few). It would be interesting to compare our construction to the existing results.

\subsection*{Generalization and future applications}
\label{rem:stable}
The setting of 2-vector spaces considered in this article is rather restrictive and doesn't include many important examples which should naturally exhibit an SSH structure, such as the category of polynomial functors over a field of positive characteristic. The Grothendieck group of this category won't be PSH, but will be a self-adjoint Hopf algebra. It seems that the correct generalization for this setting is that of stable $\infty$-categories. The main reason this seems reasonable as a generalization is that the equivalence between \comma and Beck-Chevalley squares should carry over easily to that setting. The problem with transferring everything directly to the stable $\infty$ setting is the lack of a framework for working with the $(\infty,2)$ category of all stable $\infty$ categories and computing things like \comma squares, or their appropriate generalization. The setting of stable $\infty$-categories would allow us to work with categories of complexes of objects, thus for example providing an approach to categorification of Boson-Fermion correspondence by using the connection to categorified Heisenberg action described below.

Zelevinsky's main theorem about PSH algebras is that any such algebra is isomorphic to the tensor product of many copies of the PSH algebra $\Lambda$ of symmetric functions (see \cite{Zelbook} \S2 for the precise formulation). His proof is somewhat combinatorial in nature, and gives the morphism to the tensor product only up to a non-canonical choice. In \autoref{sec:pequivariant} we construct a canonical equivalence of SSH categories $\PP_G$ and $\PP^{\otimes\Irr G}$ which descends to the isomorphism from Zelevinsky's theorem. It would be interesting to understand what analog does Zelevinsky's decomposition theorem have in general on the categorical level. As before one can hope to get especially interesting and new results outside of the semisimple setting. For example, in the case of the category of polynomial functors over a field of positive characteristic, the decomposition theorem doesn't hold on the vector space level, because the K-group lacks the positivity property. This case is interesting because of its connection to the theory of modular representations of symmetric groups. 

As we noted earlier, an SSH structure naturally defines a symmetric monoidal structure on the category. The reason for it is our choice of $\FinSet$ as the indexing category. Modifying the definition to get braided monoidal structure on the underlying category would produce a framework for dealing with interesting examples such as categories of representations of quantum groups, or of $\GL(n,\FF_q)$, and categorical objects connected to Hall algebras.

\subsection*{Funding}
This work was supported by the European Research Council [291612]; and by the Israel Science Foundation [533/14].

\subsection*{Acknowledgments} The idea of using the Beck-Chevalley condition was suggested to us by the lectures of Nick Rozenblyum in the Caesarea 2013 workshop on Grothendieck operations and $\infty$-categories. We thank him for several helpful discussions which helped shape our approach in this paper. We thank Oded Yacobi and Jiuzu Hong for introducing us to the subject of categorification and the notion of polynomial functors, Emily Riehl and Vladimir Hinich for helpful comments and Ross Street for suggesting ways to make our article more rigorous by using bicategories. We thank our colleagues Sasha Yom Din and Evgeny Musicantov for their mathematical and moral support. 

Finally, we thank our advisor Joseph Bernstein who influenced our approach to mathematics through many discussions on this and other topics.

\section{Notations}
\begin{itemize}
\item $\kk$ - a field of characteristic 0
\item $\Vect$ - the category of finite dimensional vector spaces over $\kk$
\item $\fset{n}$ - the set $\{1,2,\ldots,n\}$
\item $\pset{n}$ - the set  $\{1,2,\ldots,n, *\}$ in the category of pointed finite sets $\FinSet_*$
\item $\Sh(S)$ - the category of sheaves of finite dimensional vector spaces on $S$
\item identity morphisms will be drawn as $\longequal$ in diagrams
\item we will refer to a square of the form 
\smallttik{
A \ar{r} \ar{d} \& B \ar{d} \ar[Rightarrow,shorten <=0.7em,shorten >=0.7em]{dl}[above,sloped]{\alpha}\\
C \ar{r} \& D
}
as $\alpha$

\end{itemize}

\section{Symmetric self-adjoint Hopf categories}
\label{sec:SSHDef}
\subsection{The notion of positive self-adjoint Hopf algebra and its categorification}
In the book \cite{Zelbook} Zelevinsly defines positive self-adjoint Hopf (PSH) algebra as follows:
\begin{Definition}
\label{def:PSH}
A positive self-adjoint Hopf (PSH) algebra is a graded connected Hopf algebra over $\ZZ$ with an inner product and a distinguished finite orthogonal $\ZZ$ basis in each grade s.t. multiplication and comultiplication are adjoint and take elements with positive coefficients in the basis to elements with positive coefficients in the basis (that is to say, they are \emph{positive} maps).
\end{Definition}
Let us analyze this definition and restate it in a form which can be adapted to the category setting.

Given a PSH algebra $A$ we have an adjoint pair of multiplication and comultiplication maps which we will denote $m$ and $\Delta$:
\begin{align*}
A\otimes A \xrightarrow{m} A & {} & A\xrightarrow{\Delta}A\otimes A
\end{align*}
The Hopf axiom is the requirement that \[
\forall x,y\in A: \Delta(xy)=\Delta(x)\Delta(y)
\]
This can be restated as the commutativity of the following square:
\[
\stik{1}{
\AAS{4} \arrow[]{r}{\overline m} \arrow[leftarrow]{d}[left]{\Delta^{\otimes 2}}\& \AAS{2} \arrow[leftarrow]{d}{\Delta}\\
\AAS{2} \arrow[]{r}{m} \& A}
\]
where
\begin{align*}
\overline{m}(x\otimes y\otimes z \otimes w)&=m(x\otimes z)\otimes m(y \otimes w)\\
\Delta^{\otimes 2}(x \otimes y)&=\Delta(x)\otimes \Delta(y)
\end{align*}
Using the adjointness of $m$ and $\Delta$ this square can be obtained from the commutative square of multiplications
\begin{equation}
\label{squareHopf}
\stik{1}{} \stik{1}{
\AAS{4} \arrow[]{r}{\overline m} \arrow[]{d}[left]{m^{\otimes 2}}\& \AAS{2} \arrow[]{d}{m}\\
\AAS{2} \arrow[]{r}{m} \& A
}
\end{equation}
by replacing the verticals with their adjoints. This operation is called "taking the mate" of the square. We will say that a square satisfies the (1-categorical) \emph{Beck-Chevalley condition} if its mate commutes.

Observe that the square \eqref{squareHopf} corresponds to the Cartesian square of finite sets 
\begin{align*}
\stik{1}{
\AAS{4} \arrow[]{r}{\overline m} \arrow[]{d}[left]{m^{\otimes 2}} \& \AAS{2} \arrow[]{d}{m}\\
\AAS{2} \arrow[]{r}{m} \& A
}
& {} & \leftarrow & {} &
\stik{1}{
\fset{4} \arrow[]{r}{} \arrow[]{d}[left]{} \& \fset{2} \arrow[]{d}{} \ar[draw=none]{dl}[above,yshift=-0.5pc,xshift=-0.8pc,scale=2]{\ulcorner}\\
\fset{2} \arrow[]{r}{} \& \fset{1}
}
\end{align*}
where we denote the finite set $\{1,2,\ldots,n\}$ by $\fset{n}$. 
More generally, given the multiplication map $m:A\otimes A \rightarrow A$, we have for every map of finite sets $a:S\rightarrow T$ an obvious extension $m_{a}:\AAS{S} \rightarrow \AAS{T}$. 
The Hopf axiom implies that all squares formed using the above extension that correspond to Cartesian squares of finite sets, satisfy the Beck-Chevalley condition.

These observations allow us to reformulate the \autoref{def:PSH} as follows: Denote by $\ZAd$ the category of graded free $\ZZ$ modules with chosen finite basis in each grade and positive graded maps which admit positive graded adjoints (positivity and adjointness is considered with respect to the inner product induced by the chosen basis). Then the following holds:
\begin{Proposition}
\label{th:PSHfunctor}
A PSH algebra is the same as a functor \[
\mathcal{A}:\FinSet\xrightarrow{}\ZAd
\] which takes disjoint union to tensor product (i.e. a symmetric monoidal functor), sends morphisms in $\FinSet$ to degree zero morphisms and sends Cartesian squares to squares satisfying the Beck-Chevalley condition.
\end{Proposition}

The proof of \autoref{th:PSHfunctor} is straightforward in the view of the discussion above. Note in addition that the unit and counit maps are given by the image of the map $\emptyset \rightarrow \fset{1}$ and the connectedness follows from the fact that the image of the Cartesian square
\[
\stik{1}{
\emptyset \arrow[equal]{d} \arrow[equal]{r} \& \emptyset \ar{d}\\
\emptyset \ar{r} \& \fset{1}
}
\]
satisfies the Beck-Chevalley condition.

Given in this form, the definition of a PSH algebra can be categorified to define a class of objects which we call the \emph{symmetric self-adjoint Hopf categories}. In place of the category $\ZAd$ we will consider the bicategory $\CatAd$.

\begin{Definition} The objects of $\CatAd$ are non-negatively indexed sequences of 2-vector spaces, where a 2-vector space, as defined in \cite{2Vect}, is a semisimple $\kk$-category equivalent to a finite sum of copies of $\Vect$. We think of objects of $\CatAd$ as graded 2-vector spaces. The 1-morphisms are finite sums of exact morphisms of bounded degree and the 2-morphism are the natural transformations between those. \end{Definition}

\begin{Remark}
Note that 1-morphisms in $\CatAd$ admit left and right adjoints.
\end{Remark}


The natural analog  of the tensor of $\ZZ$-modules in $\CatAd$ is the Deligne tensor of categories from \cite{deligne}. To categorify the definition of PSH algebra we would like to replace the functor $\mathcal{A}:\FinSet\xrightarrow{}\ZAd$ with a pseudo functor of bicategories $\Hop:\FinSet \rightarrow \CatAd$ that takes the disjoint union of finite sets into the Deligne tensor of categories and sends Cartesian squares of finite sets to squares satisfying the Beck-Chevalley condition (see \autoref{sec:BC} for the general definition). Principally speaking, this definition replaces the commutative squares corresponding to the Hopf relation on the level of algebras with squares that commute up to isomorphisms forming a coherent system encoded in the functor from the category of finite sets.

In the next section we will formulate the requirement that $\Hop$ preserves tensor product in terms of it being a coCartesian functor of coCartesian fibrations of bicategories in the sense of \cite{Buckley}. This approach is motivated by the approach by J.Lurie in \cite{luriealgebra} to symmetric monoidal $(\infty,1)$-categories. Thinking about $\Hop$ as functor of fibrations leads to construction of morphisms categorifying the Heiseberg double relations in \autoref{sec:heisenbergaction}. 

\subsection{CoCartesian fibrations over the category of pointed finite sets}
Denote by $\FinSet_*$ the category of pointed finite sets. Since many objects of interest in this article are fibrations over $\FinSet_*$, we recall here some basic definitions relating to it.

Let $I$ be a pointed set. We denote $\nobase{I}$ the complement of the basepoint. A map $I\xrightarrow{\varphi} J$ in $\FinSet_*$ can be thought of as a partially defined map $\nobase{I}\rightarrow\nobase{J}$. i.e. a map $\nobase{I}\supseteq T\rightarrow \nobase{J}$.

\begin{Definition}
A map $I\rightarrow J$ is called \emph{active} if the preimage of the basepoint is the basepoint. In other words, a map is active if it is just a regular map of sets $\nobase{I}\rightarrow\nobase{J}$.
\end{Definition}

Denote by $(\FinSet_*)_{ac}$ the subcategory spanned by the active maps. There is an obvious equivalence $\FinSet\cong (\FinSet_*)_{ac}$. It is easy to see that calculating pullbacks in $(\FinSet_*)_{ac}$ is the same as calculating them in $\FinSet$. Note however that the imbedding $\FinSet\hookrightarrow \FinSet_*$ does \emph{not} preserve all limits. In particular, the final object of $\FinSet$ is not a final object in $\FinSet_*$.

\begin{Definition}[cf \cite{Buckley}]
Let $\CCC$ be a bicategory, $B$ a 1-category, and $p:\CCC\rightarrow B$ a functor.
\begin{enumerate}
\item Let $f:x\rightarrow y$ be an arrow in $B$ and $F:X\rightarrow Y$ an arrow in $\CCC$ with $p(F)=f$. Suppose that for any $Z\in\CCC$ we have that the square
\[
\stik{1}{
\CCC(Y,Z) \ar{r}{-\circ F} \ar{d} \& \CCC(X,Z) \ar{d} \\
B(p(Y),p(Z)) \ar{r}{-\circ f} \& B(p(X),p(Z))
}
\]
is a pseudo pullback of categories, then we say that $F$ is a coCartesian lift of $f$ starting at $X$.
\item We say that $p:\CCC\rightarrow B$ is a \emph{coCartesian fibration} if for any $f:x\rightarrow y$ and any $X$ such that $p(X)=x$ we have a coCartesian lift of $f$ starting at $X$.
\end{enumerate}
\end{Definition}

\begin{Remark}
In \cite{Buckley} there are further requirements on 2-cells which become trivial when the base is a 1-category. Also, the definition there is the dual one, giving Cartesian instead of coCartesian fibrations.
\end{Remark}
when working with a coCartesian fibration, we have the following slightly simpler description for coCartesian lifts following \cite{lurietopos}  Proposition 2.4.2.8:

\begin{Lemma} \label{lem:coCartInitial} Suppose that $p:\CCC\rightarrow B$ is a coCartesian fibration, then an arrow $F:X\rightarrow Y$ is a coCartesian lift over $f:x\rightarrow y$ iff it is initial in the category of arrows over $f$ with source $X$, i.e. that we can assume in the definition that $p(Z)\cong p(Y)$.
\end{Lemma}

\begin{proof}
The relevant part of the proof of \cite{lurietopos} Proposition 2.4.2.8 carries over word for word, using that coCartesian fibrations of bicategories are closed on iso-comma pullback, which is proven in \cite{Buckley} Proposition 4.3.7.
\end{proof}

The three coCartesian fibrations of interest to us are:

\begin{Definition}[$\FinDisj\xrightarrow{p}\FinSet_*$] $ $\\
\label{def:finsettensor}
\begin{itemize}
\item An object of $\FinDisj$ over $I\in\FinSet_*$ is a collection $(S_i)_{i\in I}$ of finite sets, with $\base{S}$ being a terminal object in $\FinSet$, i.e. a one point set.
\item A map $(S_i)\rightarrow(T_j)$ is a map $I\xrightarrow{\varphi}J$ in $\FinSet_*$ and a collection of functions of finite sets $S_i\rightarrow T_{\varphi(i)}$ for all $i\in I$.
\end{itemize}
\end{Definition}

\begin{Remark}
\label{rem:findisjalt}
Another way to see $\FinDisj$ is as the subcategory of arrows in $\FinSet_*$ spanned by the active maps, with the projection given by taking the target.
\end{Remark}

\begin{Definition}[$\CatAdTen\xrightarrow{p}\FinSet_*$] $ $\\
\label{SquaresInCatAd}
\label{def:CatAdTen}
\begin{itemize}
\item An object of $\CatAdTen$ over $I\in\FinSet_*$ is a collection $(\CCC_i)_{i\in I}$ of elements of $\CatAd$, with $\base{\CCC}$ being the final object in $\CatAd$, i.e. the zero category.
\item A map $(\CCC_i)\rightarrow(\DDD_j)$ over a map $I\xrightarrow{\varphi}J$ in $\FinSet_*$ is a collection of multi-exact \emph{bounded} functors ( i.e. sums of multi-exact functors of non-negative bounded degree).
\[
f_j:\prod_{\varphi(i)=j}\CCC_i\rightarrow \DDD_j
\]
\begin{Note}If $\varphi^{-1}(j)=\emptyset$, then we should give a map from the empty product, which is defined as $\Vect$, to $\DDD_j$. This is the same as specifying an object in $\DDD_j$.
\end{Note}
\item a 2-cell in $\CatAdTen$ between two maps $F,G:(\CCC_i)\rightarrow (\DDD_j)$ over $\varphi:I\rightarrow J$ is a collection of 2-morphisms in $\CatAd$ between the functors defining $F,G$, i.e. natural transformations $\varphi_j:F_j\rightarrow G_j$. 
\begin{Note}
When $\varphi^{-1}(j)=\emptyset$, then giving $\alpha_j:F_j\rightarrow G_j$ amounts to specifying a morphism between the specified objects.
\end{Note}
\end{itemize}
\end{Definition}

\begin{Definition}[$\ZAdTen\xrightarrow{p}\FinSet_*$] $ $\\
\label{def:ZAdTen}
\begin{itemize}
\item An object of $\ZAdTen$ over $I\in\FinSet_*$ is a collection $(V_i)_{i\in I}$ of free graded $\ZZ$ modules with a chosen finite basis in each degree, with $\base{V}$ being the 0 module.
\item A map $(V_i)\rightarrow(U_j)$ over a map $I\xrightarrow{\varphi}J$ in $\FinSet_*$ is a collection of multilinear \emph{bounded} (as above) maps
\[
f_j:\prod_{\varphi(i)=j}V_i\rightarrow U_j
\]
\begin{Note}If $\varphi^{-1}(j)=\emptyset$, this means specifying a map from $\ZZ$ to $U_j$, or just an element in $U_j$.
\end{Note}
\end{itemize}
\end{Definition}


\subsubsection{coCartesian lifts}
For simplicity let us consider what are coCartesian lifts over the unique active map $I\rightarrow\pset{1}$.

Using \autoref{lem:coCartInitial}, in $\FinDisj$ this is a collection of maps $S_i\rightarrow T$ which is initial among collections of maps $S_i\rightarrow T'$, i.e. a presentation of $T$ as the disjoint union (or coproduct) of the $S_i$. In the point of view of \autoref{rem:findisjalt} this is just a map $S\rightarrow T$ initial among all maps from $S$, so just an isomorphism. Clearly a lift exists for any map in $\FinSet_*$.

In $\CatAdTen$ this is a multi-exact graded functor $\prod\CCC_i\rightarrow \DDD$ with the property that
\[
\stik{1}{
\CatAdTen(D,E) \ar{r} \ar{d} \& \CatAdTen((C_i),E) \ar{d} \\
\FinSet_*(\pset{1},\pset{1}) \ar{r} \& \FinSet_*(I,\pset{1})
}
\]
is a pseudo pullback, which in this case just means the upper arrow is an equivalence.

This is precisely what it means to present $\DDD$ as the (graded) Deligne tensor of the $\{\CCC_i\}$ (see \cite{deligne}). From \cite{lopez} it follows that in $\CatAd$ the Deligne tensor always exists and so $\CatAdTen$ is a coCartesian fibration. Note also that in $\ZAdTen$ this is a presentation of a module as a tensor of several modules.

In all of our fibrations all coCartesian lifts can be constructed from the ones above. The only non-obvious case is the lift of the map $\pset{0}\rightarrow\pset{1}$ starting at the empty list. In $\FinDisj$ this is the map $()\rightarrow(\emptyset)$. In $\CatAdTen$ by definition a map from the empty list to a category is a map from $\Vect$ to that category, and so the lift is $()\rightarrow(\Vect)$ given by the identity map of $\Vect$. Similarly, note that in $\ZAdTen$ this is $()\rightarrow(\ZZ)$ given by the identity of $\ZZ$.

We will sometimes write $\boxtimes$ to denote coCartesian arrows in $\CatAdTen$ or $\ZAdTen$, and $+$ to denote coCartesian arrows in $\FinDisj$.


\subsection{Hopf categories}

We can now formulate our main definition

\begin{Definition}
\label{maindef}
A \emph{symmetric self-adjoint Hopf }(SSH) category is a morphism of fibrations $\FinDisj\xrightarrow{\Hop}\CatAdTen$, which
\begin{enumerate}
\item Takes maps in $\FinDisj$ to degree zero maps in $\CatAdTen$.
\item Is coCartesian, i.e. takes coCartesian arrows to coCartesian arrows.
\item Takes Cartesian squares in the fiber over $\pset{1}$ to squares satisfying the Beck-Chevalley condition.
\end{enumerate}
\end{Definition}

We will show in \autoref{sec:2Cartesian} that the Beck-Chevalley squares in $\CatAd$ can be viewed as a certain 2-categorical generalization of Cartesian squares that we call \comma squares. With this generalization in mind we can think of the functor $\Hop$ as preserving \comma squares and we will refer to this property as "the continuity" of $\Hop$.

We will say that $\Hop$ gives a structure of SSH category on $\Hop(\fset{1})$ and that $\Hop(\fset{1})$ is an SSH category.

Explicitly, Specifying the SSH category structure amounts to the following data:
\begin{itemize}
\item For every finite set $S$, a category $\Hop(S)$.
\item For every finite collection of maps of finite sets $(S_i\xrightarrow{\varphi_i}T)_{i\in I}$, a functor $\Hop(\varphi_i):\prod_i \Hop(S_i)\rightarrow \Hop(T)$.
\item For every commutative square 
\[
\stik{1}{
(S_i)_{i\in I} \arrow{d}\arrow{r} \& (Q_k)_{k\in K} \ar{d}\\
(T_j)_{j\in J} \arrow{r} \& U
}
\]
(where commutativity is understood separately for each $i$), a 2-commutative square 
\[
\stik{1}{
\Hop((S_i)_{i\in I}) \arrow{d}\arrow{r} \& \Hop((Q_k)_{k\in K}) \ar{d} \ar[Rightarrow,shorten <=0.7em,shorten >=0.7em]{dl}[above,sloped]{\sim}\\
\Hop((T_j)_{j\in J}) \arrow{r} \& \Hop(U)
}
\]

\end{itemize}

And it has to satify the following properties: \begin{itemize}
\item $3$-cells in $\FinDisj$ map to $3$-cells in $\CatAdTen$, e.g. commutative cubes go to commutative cubes.
\item coCartesian arrows map to coCartesian arrows.
\item Cartesian squares of sets in the fiber over $\pset{1}$ map to squares of categories satisfying the Beck-Chevalley condition.
\end{itemize}
\begin{Notation}
\label{fiberover1}
The restriction of the functor $\Hop$ to the fiber over $\pset{1}$ gives a functor that we will denote $\Hop^1:\FinSet \rightarrow \CatAd$ which takes Cartesian squares to Beck-Chevalley squares. One can think of the images of the maps in $\FinSet$ under $\Hop^1$ and their adjoints as the categorical analogs of the maps of multiplication and comultiplication in the Hopf algebra (for any number of variables). Hence we will denote the image of the arrows $S \xrightarrow{a} T$ by $m_a$ and the image of the arrow $\fset{2}\rightarrow \fset{1}$ will be denoted just by $m$. 

Let $\CCC=\Hop(\fset{1})$ (in other words $\Hop$ gives an SSH structure on the category $\CCC$). Then elsewhere in the article we will denote the image of a finite set $U$ under $\Hop$ by $\CS{U}$. This notation is motivated by the fact that, as we noted earlier, the image of $U$ should satisfy the universal property of being a Deligne tensor. 
\end{Notation}

\subsection{Basic properties}
\label{ssec:SSHbasicproperties}

Throughout, let $\Hop$ be a SSH structure on $\CCC$.

\begin{Proposition}
\label{frenkelcranehopfcat}
$\Hop$ induces on $\CCC$ the structure of a Hopf category in the sense of Frenkel and Crane from \cite{FrenkelCrane} (we recall the definition in the proof).
\end{Proposition}

\begin{proof}
The data of a Hopf category is the following:
\begin{itemize}
\item Functors $m:\CCC\otimes\CCC\rightarrow \CCC$ and $\Delta:\CCC\rightarrow\CCC\otimes\CCC$ which are (co)associative (i.e. monoidal and comonoidal structures on $\CCC$).
\item A 2-isomorphism
\[
\widestik{1}{
\CCC^{\otimes\fset{4}} \arrow[shorten >=0.4cm,shorten <=0.4cm,Rightarrow]{dr}[above,sloped]{\sim} \ar{r}{m_{13}\otimes m_{24}} \ar[leftarrow]{d}[left]{\Delta\otimes \Delta}  \& \CCC^{\otimes\fset{2}} \ar[leftarrow]{d}{\Delta} \\
\CCC^{\otimes\fset{2}} \ar{r}[below]{m} \& \CCC
}
\]

satisfying the relations imposed by requiring that the following cubes commute (abbreviating $n$ for $\CCC^{\otimes\fset{n}}$ and e.g $1,(23)$ for $\Id_1\otimes m_{23}$ and $\overline{1,(23)}$ for $\Id_1\otimes \Delta_{23}$):
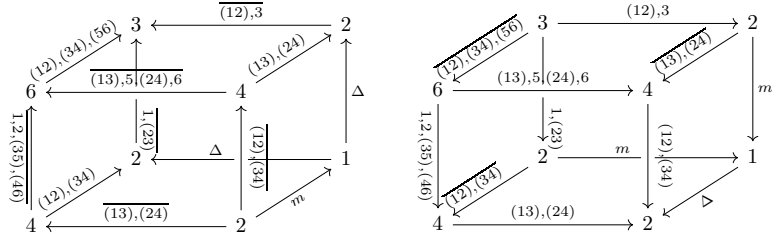
\begin{align*}
\widestik{0.8}{
{} \& 3 \ar[leftarrow]{rr}{\overline{(12),3}} \ar[leftarrow]{dd}[sloped,xshift=0.5pc,yshift=0.5pc]{\overline{1,(23)}} \& {} \& 2 \ar[leftarrow]{dd}{\Delta}\\
6 \ar[crossing over,leftarrow]{rr}{\overline{(13),5,(24),6}} \ar[leftarrow]{dd}[left,sloped,yshift=-0.5pc,xshift=2pc]{\overline{1,2,(35),(46)}} \ar{ur}[sloped,xshift=2pc]{(12),(34),(56)} \& {} \& 4 \ar[leftarrow]{dd}[sloped,yshift=0.7pc,xshift=-1.5pc]{\overline{(12),(34)}} \ar{ur}[sloped,xshift=1pc]{(13),(24)} \& {}\\
{} \& 2 \ar[leftarrow]{rr}[xshift=-1pc]{\Delta} \& {} \& 1\\
4 \ar[leftarrow]{rr}{\overline{(13),(24)}} \ar{ur}[sloped,xshift=1pc]{(12),(34)} \& {} \& 2 \ar{ur}[below,sloped]{m}
\latearrow{commutative diagrams/crossing over}{4-3}{2-3}{}
} & {} & \widestik{0.8}{
{} \& 3 \ar{rr}{(12),3} \ar{dd}[sloped,xshift=0.5pc,yshift=0.5pc]{1,(23)} \& {} \& 2 \ar{dd}{m}\\
6 \ar[crossing over]{rr}{(13),5,(24),6} \ar{dd}[left,sloped,yshift=-0.5pc,xshift=2pc]{1,2,(35),(46)} \ar[leftarrow]{ur}[sloped,xshift=2pc]{\overline{(12),(34),(56)}} \& {} \& 4 \ar{dd}[sloped,yshift=0.7pc,xshift=-1.5pc]{(12),(34)} \ar[leftarrow]{ur}[sloped,xshift=1pc]{\overline{(13),(24)}} \& {}\\
{} \& 2 \ar{rr}[xshift=-1pc]{m} \& {} \& 1\\
4 \ar{rr}{(13),(24)} \ar[leftarrow]{ur}[sloped,xshift=1pc]{\overline{(12),(34)}} \& {} \& 2 \ar[leftarrow]{ur}[below,sloped]{\Delta}
\latearrow{commutative diagrams/crossing over}{2-3}{4-3}{}
}
\end{align*}
\end{itemize}

The 2-isomorphism is obtained as the left mate of $\Hop$ of the cartesian square of finite sets
\[
\stik{1}{
\fset{4} \ar{r}{} \ar{d}{} \& \fset{2} \ar{d}{}\\
\fset{2} \ar{r}{} \& \fset{1}
}
\]
The relation cubes commute since they are both mates of $\Hop$ of the cube of sets (see \autoref{sec:BC} for details regarding mates of cubes and squares)
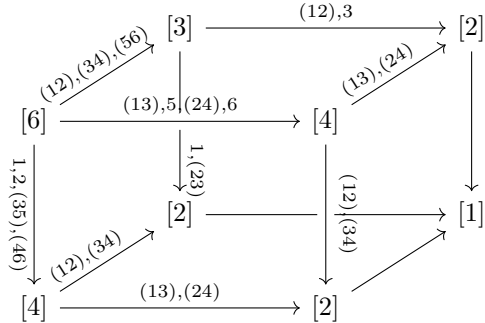
\[
\widestik{1}{
{} \& \fset{3} \ar{rr}{(12),3} \ar{dd}[sloped,xshift=0.5pc,yshift=0.5pc]{1,(23)} \& {} \& \fset{2} \ar{dd}{}\\
\fset{6} \ar[crossing over]{rr}{(13),5,(24),6} \ar{dd}[left,sloped,yshift=-0.5pc,xshift=2pc]{1,2,(35),(46)} \ar{ur}[sloped,xshift=2pc]{(12),(34),(56)} \& {} \& \fset{4} \ar{dd}[sloped,yshift=0.7pc,xshift=-1.5pc]{(12),(34)} \ar{ur}[sloped,xshift=1pc]{(13),(24)} \& {}\\
{} \& \fset{2} \ar{rr}[xshift=-1pc]{} \& {} \& \fset{1}\\
\fset{4} \ar{rr}{(13),(24)} \ar{ur}[sloped,xshift=1pc]{(12),(34)} \& {} \& \fset{2} \ar{ur}[below,sloped]{}
\latearrow{commutative diagrams/crossing over}{2-3}{4-3}{}
}
\]
where e.g. by $1,2,(35),(46)$ we mean the map $1\mapsto 1;2\mapsto 2;3,5\mapsto 3;4,6\mapsto 4$.

The left one is obtained by taking the left mate twice, and the right one by taking the right mate. All faces end up invertible because they are all either mates of $\Hop$ of cartesian squares or $\Hop$ of squares, or double mates of $\Hop$ of squares. 
\end{proof}

\begin{Remark}
The cube in the above proof can be recovered as follows:
\begin{enumerate}
\item Consider the associator square:
\[
\stik{1}{
\fset{3} \ar{d}[left,sloped,yshift=-0.5pc,xshift=1pc]{1,(23)} \ar{r}{(12),3} \& \fset{2} \ar{d} \\
\fset{2} \ar{r} \& \fset{1}
}
\]
\item Add the map $\fset{2}\rightarrow\fset{1}$ in a transversal direction:
\[
\stik{1}{
{} \& \fset{3} \ar{rr}{(12),3} \ar{dd}[sloped,xshift=-1.5pc,yshift=0.5pc]{1,(23)} \& {} \& \fset{2} \ar{dd}{}\\
{} \& {} \& {} \& {}\\
{} \& \fset{2} \ar{rr}[xshift=-1pc]{} \& {} \& \fset{1}\\
{} \& {} \& \fset{2} \ar{ur}[below,sloped]{}
}
\]
\item Form the pullbacks on the right and bottom:
\[
\widestik{1}{
{} \& \fset{3} \ar{rr}{(12),3} \ar{dd}[sloped,xshift=-1.5pc,yshift=0.5pc]{1,(23)} \& {} \& \fset{2} \ar{dd}{}\\
{} \& {} \& \fset{4} \ar{dd}[sloped,yshift=0.7pc,xshift=-1.5pc]{(12),(34)} \ar{ur}[sloped,xshift=1pc]{(13),(24)} \& {}\\
{} \& \fset{2} \ar{rr}[xshift=-1pc]{} \& {} \& \fset{1}\\
\fset{4} \ar{rr}{(13),(24)} \ar{ur}[sloped,xshift=1pc]{(12),(34)} \& {} \& \fset{2} \ar{ur}[below,sloped]{}
\latearrow{commutative diagrams/crossing over}{2-3}{4-3}{}
}
\]
\item Form the pullback cube.
\end{enumerate}
\end{Remark}

\begin{Proposition}
$\Hop(\emptyset)$ is canonically equivalent to $\Vect$.
\end{Proposition}

\begin{proof}
Let $\alpha:\pset{0}\rightarrow\pset{1}$ be the unique map. In $\CatAdTen$, over $\pset{0}$, there is only the list $(\base{C})$ consisting of the trivial category. Therefore, a map in $\CatAdTen$ over $\alpha$ is a map $(\base{C})\rightarrow (\base{C},D)$, which amounts to a functor from the empty product to $D$, i.e. the specification of an object in $D$.

We have:
\begin{enumerate}
\item The arrow $(\base{S})\rightarrow (\base{S},\emptyset)$ in $\FinDisj$ is a coCartesian arrow over $\alpha$.
\item The arrow $(\base{C})\rightarrow (\base{C},\Vect)$ in $\CatAdTen$ which sends the empty product to  $\kk$, is a coCartesian arrow over $\alpha$.
\end{enumerate}

The first is obvious, and second is just saying that a choice of an element in $D$ is the same as a functor $\Vect\rightarrow D$. 

As a result, the first arrow must go to an arrow which is canonically equivalent to the second arrow and in particular $\Hop(\emptyset)$ is canonically equivalent to $\Vect$.
\end{proof}

\begin{Proposition}
$\Hop$ defines a symmetric monoidal structure on $\CCC$.
\label{prop:symmonstruct}
\end{Proposition}

\begin{proof}
Consider the section $s$ of $\FinDisj\rightarrow\FinSet_*$ given by $\pset{n}\mapsto(\fset{1},\ldots,\fset{1})$ which sends a map $\alpha:\pset{m}\rightarrow\pset{n}$ to the appropriate collection of identity maps. Composing this section with the map $\FinDisj\rightarrow\CatAdTen$ we get a section $\bar{s}$ of $\CatAdTen\rightarrow\FinSet_*$, which gives a symmetric monoidal structure on $\bar{s}(\fset{1})=H(\fset{1})=\CCC$. 
\end{proof}

Explicitly, this monoidal structure is given by applying $\Hop$ to the map $(\fset{1},\fset{1})\rightarrow (\fset{1})$ given by the two copies of $\Id: \fset{1} \rightarrow \fset{1}$ which lies over the active map $\pset{2}\rightarrow\pset{1}$.

Let us denote this map by $(F,G)\mapsto F\otimes G$, where $F,G$ are objects in $\CCC$
\begin{Remark}
The map $(F,G)\mapsto F\otimes G$ factors essentially uniquely as $\Hop$ applied to the composition
\[
\nnstik{1}{
\fset{1} \ar{dr} \\
\& \fset{2} \ar{r} \& \fset{1} \\
\fset{1} \ar{ur}
}
\]
The factorization is given on $F,G \in \CCC$ as $(F,G)\mapsto F\boxtimes G \mapsto m(F\boxtimes G)$ i.e. a coCartesian arrow followed by an arrow over $\Id_{\pset{1}}$.
\end{Remark}

\begin{Corollary}
\label{SSH-kgroup}
The functor $\Hop$ canonically defines the structure of a PSH structure on the Grothendieck group $K(\CCC)$.
\end{Corollary}

\begin{proof}
Let $A:=K(\CCC)$.

Denote by $(\CatAdTen)^0$ the same bicategory where all non-invertible 2-morphisms have been discarded.

Then we have a functor $(\CatAdTen)^0\xrightarrow{K(-)}\ZAdTen$ which obviously sends Deligne tensor to tensor of $\ZZ$-modules.


The monoidal structure on $\CCC$ yields a structure of algebra on $A$. The Hopf axiom comes from the continuity of the original functor, as follows:

Consider the Cartesian square in $\FinDisj$, and its image under $\Hop$
\[
\stik{1}{
\fset{4} \ar{r}{a} \ar{d}{b} \& \fset{2} \ar{d}{c} \\
\fset{2} \ar{r}{d} \& \fset{1}
}\mapsto
\stik{1}{
\Hop(\fset{4}) \ar{r}{\Hop(a)} \ar{d}[left]{\Hop(b)} \& \Hop(\fset{2}) \ar{d}{\Hop(c)} \ar[Rightarrow,shorten <=1pc,shorten >=1pc]{dl}[above,sloped]{\alpha} \\
\Hop(\fset{2}) \ar{r}{\Hop(d)} \& \Hop(\fset{1})
}
\]
Since the square is Cartesian, its image satisfies the Beck-Chevalley condition. Taking $K$ groups, we get the commutative square of algebras 
\[
\widestik{1}{A^{\otimes 4} \ar{r}{m_{12}\otimes m_{34}} \ar{d}[left]{m_{13}\otimes m_{24}} \& A^{\otimes 2} \ar{d}{m}\\
A^{\otimes 2} \ar{r}{m} \& A
}
\]
The Hopf axiom amounts to showing that this square with the top and bottom arrows replaced by adjoints should commute. But replacing top and bottom with adjoints in the square of modules is the same as taking the left mate of the square $\alpha$ and then taking $K$ groups. The square $\alpha$ satisfies the Beck-Chevalley condition, hence its left mate has an isomorphism in the middle, so we conclude that the square of modules commutes.

Regarding connectedness of $K(\CCC)$ see the discussion in \autoref{rem:connected}.
\end{proof}

\begin{Proposition}
\label{prop:shonpower}
For any finite set $T_0$ we have a natural SSH structure on $\Hop(T_0)$.
\end{Proposition}

\begin{proof}
Define a functor by $S\mapsto \Hop(S\times T_0)$, then it is obviously an SSH functor, and it sends $\fset{1}$ to $\Hop(T_0)$.
\end{proof}

\subsection{The categorical analogs of Hopf algebra structures}

\subsubsection{Connectedness}
\label{rem:connected}
\label{connectedness}

Recall that a PSH algebra $A$ is defined to be a \emph{connected} Hopf algebra, namely it has unit and counit morphisms $\ZZ\rightarrow A_0$ and $A\rightarrow\ZZ$ such that they give an isomorphism of $A_0$ with $\ZZ$.\\
This is categorified as follows:
The unit and counit in $K(\CCC)$ are the image of the map $\emptyset\rightarrow 1$ in $\FinDisj$ and its left adjoint. Denote the image of this map under $\Hop$ by $m_\emptyset$ and its adjoint by $\Delta_\emptyset$.

Consider the diagrams
\begin{align*}
\stik{1}{
\emptyset \arrow[equal]{d} \arrow[equal]{r} \& \emptyset \ar{d}\\
\emptyset \ar{r} \& 1
}
& {\stik{1}{}} &
\stik{1}{
\emptyset \arrow{d} \arrow{r} \& 1 \ar[equal]{d}\\
1 \ar[equal]{r} \& 1
}
\end{align*}

They go to squares in $\CatAd$

\begin{align*}
\stik{1}{
\Vect \arrow[equal]{d} \arrow[equal]{r} \& \Vect \ar{d}{m_\emptyset}\ar[Rightarrow,shorten <=0.7pc,shorten >=0.7pc]{dl}[above, sloped]{\Id}\\
\Vect \ar{r}[below]{m_\emptyset} \& \CCC
}
& {\stik{1}{}} &
\stik{1}{
\Vect \arrow{d}[left]{m_\emptyset} \arrow{r}{m_\emptyset} \& \CCC \ar[equal]{d} \ar[Rightarrow,shorten <=0.7pc,shorten >=0.7pc]{dl}[above, sloped]{\Id}\\
\CCC \ar[equal]{r} \& \CCC
}
\end{align*}
their left mates are
\begin{align*}
\stik{1}{
\Vect \arrow[equal]{d} \arrow[equal]{r} \& \Vect \ar{d}{m_\emptyset}\\
\Vect \ar[leftarrow]{r}[below]{\Delta_\emptyset} \& \CCC \ar[Rightarrow,shorten <=0.7pc,shorten >=0.7pc]{ul}[above, sloped]{\alpha}
}
& {\stik{1}{}} &
\stik{1}{
\Vect \arrow{d}[left]{m_\emptyset} \arrow[leftarrow]{r}{\Delta_\emptyset} \& \CCC \ar[equal]{d}\\
\CCC \ar[equal]{r} \& \CCC \ar[Rightarrow,shorten <=0.7pc,shorten >=0.7pc]{ul}[above, sloped]{\beta}
}
\end{align*}
The left square came from an \comma square, so contains an isomorphism, i.e. $\alpha$ is invertible, but $\beta$ need not be, and in a non trivial situation will not be.

Define $\CCC_00$ to be the full subcategory of objects $X\in\CCC$ for which $\beta:X\rightarrow m_\emptyset\Delta_\emptyset X$ is invertible. Note that all objects of $\CCC_00$ are necessarily of degree 0 because the map $m_\emptyset$ lands in degree 0.

\begin{Definition}
We say that an SSH category is \emph{connected} if $\CCC_00=\CCC_0$.
\end{Definition}

\begin{Proposition}
$m_\emptyset,\Delta_\emptyset$ give an adjoint equivalence between $\Vect$ and $\CCC_00$.
\end{Proposition}
\begin{proof}
All we need to show is that $m_\emptyset$ always lands in $\CCC_0$. This holds because $\alpha,\beta$ are just the counit and unit of the adjunction $(m_\emptyset\dashv \Delta_\emptyset)$, so the composition\[
m_\emptyset V\xrightarrow{\beta}m_\emptyset\Delta_\emptyset m_\emptyset V\xrightarrow{\alpha}m_\emptyset V
\]
is the identity of $m_\emptyset V$. In particular $\beta$ is invertible for $m_\emptyset V$.
\end{proof}

\begin{Corollary}
The $K$-group of a connected SSH category is a (connected) PSH algebra.
\end{Corollary}

\subsubsection{The categorical Hopf axiom}
In the notation of \autoref{fiberover1}, for any \comma square of maps of sets we have the corresponding diagram in $\CatAd$
\[
\tik
S\ar{r}{a} \ar{d}{c} &T \ar{d}{b} \\ 
R \ar{r}{d} & U 
\tak
\rightsquigarrow
\tik
\CS{S}\arrow[]{d}{m_c} \arrow[]{r}{m_a} & \CS{T} \arrow[]{d}{m_b} \arrow[shorten >=0.4cm,shorten <=0.4cm,Rightarrow]{dl}[above,sloped]{\sim} \\ 
\CS{R} \arrow[]{r}{m_d} & \CS{U}
\tak
\]
(Here $S=T\times_U R$). Denoting the left adjoint to $m$ by $\Delta^l$ we can consider the left mate of this square
\[
\tik
\CS{S} \arrow[shorten >=0.4cm,shorten <=0.4cm,Leftarrow]{dr}[above,sloped]{} \arrow[leftarrow]{r}{\Delta^l_a} & \CS{T} \\ 
\CS{R} \arrow[leftarrow]{u}{m_c} \arrow[leftarrow]{r}{\Delta^l_d} & \CS{U} \arrow[leftarrow]{u}{m_b}
\tak
\]

And the Beck-Chevalley condition tells us that the $2$-morphism in this square is invertible. This can be viewed as a natural categorification of the Hopf axiom in the algebra "for any number of variables". 

This system of isomorphisms includes a compatibility with the monoidal structure as we saw in the proof of \autoref{frenkelcranehopfcat}.

Note that taking the right adjoint of $m$ gives an additional system of "Hopf isomorphisms". This phenomenon appears in the article \cite{savageyacobi} where the authors construct a dual pair of Hopf algebras from a monoidal category. We suggest that the underlying structure of their construction should be a non-semisimple selfadjoint Hopf category. See also \autoref{rem:leftright}.


\subsection{\comma squares}
\label{sec:2Cartesian}

A \comma square in a $2$-category can be thought of as a square \smallttik{
A \ar{r} \ar{d} \& B \ar{d} \ar[Rightarrow,shorten <=0.7em,shorten >=0.7em]{dl} \\
C \ar{r} \& D
} 
which is final in the category of squares which share its right and bottom sides. This is a slight weakening of the notion of comma square - the precise definition is given in \autoref{app:WeightedLimits}.

Morally, it means that for any square \smallttik{
E \ar{r} \ar{d} \& B \ar{d} \ar[Rightarrow,shorten <=0.7em,shorten >=0.7em]{dl}\\
C \ar{r} \& D
}
there exists an "essentially unique" oriented commutative cube 
\begin{equation}
\label{commacube}
\stik{0.7}{
{} \& C \ar{rr}[name=topback,below]{} \ar[equal]{dd}[name=backleft,sloped,xshift=-1.5pc]{} \& {} \& D \ar[equal]{dd}\\
E \ar[crossing over]{rr}[name=topfront,above]{}\ar{dd}[name=frontleft]{} \ar{ur}[name=topleft,below]{} \& {} \& B \ar[equal]{dd}[name=frontright,sloped,xshift=1pc]{} \ar{ur}[name=topright,below]{} \& {}\\
{} \& C \ar{rr} \& {} \& D\\
A \ar{rr}[name=bottomfront,below]{} \ar{ur}[name=bottomleft,sloped,yshift=-0.2pc]{} \& {} \& B \ar{ur}
\arrow[Rightarrow,to path={(topright) to[bend left] (topback)}]{}{}
\arrow[Rightarrow,to path={(frontright) to[bend right] (bottomfront)}]{}{}
\arrow[equal,crossing over,to path={(frontright) to[bend right] (topright)}]{}{}
\arrow[Rightarrow,to path={(bottomfront) to[bend right] (bottomleft)}]{}{}
\arrow[Rightarrow,to path={(bottomleft) to[bend left] (backleft)}]{}{}
\arrow[equal,to path={(topback) to[bend left] (backleft)}]{}{}
\latearrow{/tikz/commutative diagrams/crossing over,/tikz/commutative diagrams/equal}{2-3}{4-3}{}
\latearrow{/tikz/commutative diagrams/crossing over}{2-1}{2-3}{}
}
\end{equation}
where vertices are objects, edges are $1$-morphisms and faces are $2$-morphisms composed as indicated by the arrows in the diagram (more on the subject of commutative cubes in \autoref{app:commcube}).

In $\FinSet$, or any 1-category, an \comma square is the same as a usual Cartesian square.

\begin{Proposition}
\label{prop:bccart}
\label{proof:cartesianBC}
In $\CatAd$ a square is \comma iff it satisfies the Beck-Chevalley condition.
\end{Proposition}

\begin{proof}
For the definitions and details regarding the Beck-Chevalley condition see \autoref{sec:BC}.

Let \[
\stik{1}{
A \ar{r}{h} \ar{d}{i} \& B \ar{d}{g} \ar[Rightarrow,shorten <=0.7em,shorten >=0.7em]{dl}[above,sloped]{\alpha}\\
C \ar{r}{f} \& D
}
\]
be a square, and suppose it is \comma. Denote the $2$-morphism in its left mate by $\alpha_L$. We want to show that $\alpha_L$ is invertible. For any functor $x\in\CatAd$, denote by $x_L,x_R$ the left and right adjoints of $x$, and consider the square 
\[
\stik{1}{
D \ar[leftarrow]{r}{g} \ar[equal,shorten <=0.7em,shorten >=0.7em]{dr}[above,sloped]{} \ar{d}{f_L} \& B \ar{d}{g} \\
C \ar[leftarrow]{r}{f_L} \& D
}
\]
Take its right mate to get the square
\[
\stik{1}{
D \ar{r}{g_R}  \ar{d}{f_L} \& B \ar[Rightarrow,shorten <=0.7em,shorten >=0.7em]{dl}[above,sloped]{\beta} \ar{d}{g} \\
C \ar{r}{f} \& D
}
\]
which has the same right and bottom sides as our original square. As a consequence we have a commutative cube
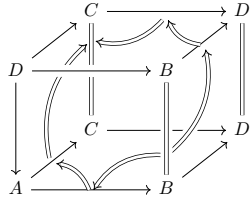
\[
\stik{0.7}{
{} \& C \ar{rr}[name=topback,below]{} \ar[equal]{dd}[name=backleft,sloped,xshift=-1.5pc]{} \& {} \& D \ar[equal]{dd}\\
D \ar[crossing over]{rr}[name=topfront,above]{}\ar{dd}[name=frontleft]{} \ar{ur}[name=topleft,below]{} \& {} \& B \ar[equal]{dd}[name=frontright,sloped,xshift=1pc]{} \ar{ur}[name=topright,below]{} \& {}\\
{} \& C \ar{rr} \& {} \& D\\
A \ar{rr}[name=bottomfront,below]{} \ar{ur}[name=bottomleft,sloped,yshift=-0.2pc]{} \& {} \& B \ar{ur}
\arrow[Rightarrow,to path={(topright) to[bend left] (topback)}]{}{}
\arrow[Rightarrow,to path={(frontright) to[bend right] (bottomfront)}]{}{}
\arrow[Rightarrow,crossing over,to path={(frontright) to[bend right] (topright)}]{}{}
\arrow[Rightarrow,to path={(bottomfront) to[bend right] (bottomleft)}]{}{}
\arrow[Rightarrow,to path={(bottomleft) to[bend left] (backleft)}]{}{}
\arrow[Rightarrow,to path={(topback) to[bend left] (backleft)}]{}{}
\latearrow{/tikz/commutative diagrams/crossing over,/tikz/commutative diagrams/equal}{2-3}{4-3}{}
\latearrow{/tikz/commutative diagrams/crossing over}{2-1}{2-3}{}
}
\]

Taking the left mate of this cube as described in \autoref{lem:cubemate}, we get a commutative cube that on one side has the composition of three identity morphisms and on the other side has a composition of the form $\gamma_1\circ \gamma_2\circ\alpha_L$. In particular, $\alpha_L$ is invertible, as required.

In the other direction, suppose that our originial square $\alpha$ satisfies the Beck-Chevalley condition. We want to show it is \comma using the definitions in \autoref{app:PBcont}. Let $X$ be the category of squares with bottom right corner \stik{0.7}{\& B \ar{d}{g}\\ C \ar{r}{f} \& D} as described in \autoref{app:WeightedLimits}. We need to show that the map $X\rightarrow\point$ has a lax right adjoint given by our square.

Consider a square 
\[
\stik{1}{
M \ar{r}{b} \ar{d}{a} \& B \ar{d}{g} \ar[Rightarrow,shorten <=0.7em,shorten >=0.7em]{dl}[above,sloped]{\beta}\\
C \ar{r}{f} \& D
}
\]
(for simplicity we have taken a square with precisely the same corner and not the more general case described in \autoref{app:WeightedLimits}, but this does not affect the proof).

Denote $S=\Hom_X(\beta,\alpha)$. We will show that $S$ has a final object (as discussed in \autoref{app:PBcont}).

Let $C_R=i_R\circ a$ and consider the following object in $S$ which we also denote $C_R$:
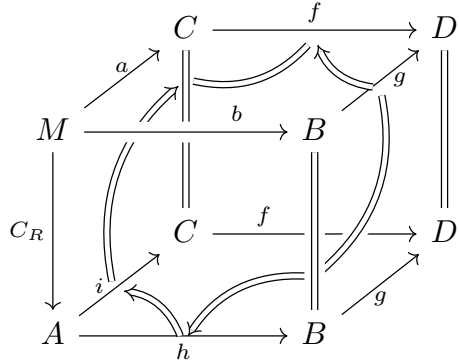
\[
\stik{1.2}{
{} \& C \ar{rr}[name=topback,below]{}[above]{f} \ar[equal]{dd}[name=backleft,sloped,xshift=-1.5pc]{} \& {} \& D \ar[equal]{dd}\\
M \ar[crossing over]{rr}[name=topfront,above]{}[above,near end,yshift=1pt]{b} \ar{dd}[name=frontleft]{}[left]{C_R} \ar{ur}[name=topleft,below]{}[above]{a} \& {} \& B \ar[equal]{dd}[name=frontright,sloped,xshift=1pc]{} \ar{ur}[name=topright,below]{}[below,near end]{g} \& {}\\
{} \& C \ar{rr}[above,near start]{f} \& {} \& D\\
A \ar{rr}[name=bottomfront,below]{}[below]{h} \ar{ur}[name=bottomleft,sloped,yshift=-0.2pc]{}[above,near start]{i} \& {} \& B \ar{ur}[below]{g}
\arrow[Rightarrow,to path={(topright) to[bend left] (topback)}]{}{}
\arrow[Rightarrow,to path={(frontright) to[bend right] (bottomfront)}]{}{}
\arrow[equal,crossing over,to path={(frontright) to[bend right] (topright)}]{}{}
\arrow[Rightarrow,to path={(bottomfront) to[bend right] (bottomleft)}]{}{}
\arrow[Rightarrow,to path={(bottomleft) to[bend left] (backleft)}]{}{}
\arrow[equal,to path={(topback) to[bend left] (backleft)}]{}{}
\latearrow{/tikz/commutative diagrams/crossing over,/tikz/commutative diagrams/equal}{2-3}{4-3}{}
\latearrow{/tikz/commutative diagrams/crossing over}{2-1}{2-3}{}
}
\]
where the left face is given by the counit $\epsilon$ of the adjunction $(i\dashv i_R)$ and the front face is given by the composition (reading from left to right)
\[
g_R:=\nstik{1}{
{} \ar{d}{b} \& \ar{d}{b} \ar[yshift=-1em,shorten <=1em,shorten >=0.2em]{dr}[xshift=0em]{\beta} \& \ar{d}{a} \& \ar{d}{a}\\
{} \ar[equal]{d} \ar[yshift=-1em,shorten <=1em,shorten >=0.2em]{dr}[xshift=-0.5em]{\eta} \& \ar{d}{g} \ar[yshift=-1.5em,shorten <=1.2em]{ur} \& \ar{d}{f} \ar[yshift=-1em,shorten <=1em,shorten >=0.2em]{dr}[xshift=0em]{\alpha_R^{-1}} \& \ar{d}{i_R}\\
{} \ar[equal]{d} \ar[yshift=-1.5em,shorten <=1.2em]{ur} \& \ar{d}{g_R} \& \ar{d}{g_R} \ar[yshift=-1.5em,shorten <=1.2em]{ur} \& \ar{d}{h}\\
{} \& {} \& {} \& {}
}
\]
It can be easily checked that this cube commutes.

\begin{Claim}
The cube $C_R$ is final in $S$.
\end{Claim}

\begin{Note}
In addition to the cube we constructed above, we should also give a morphism between the compositions in the corner as described in the end of \autoref{app:WeightedLimits} which in this case can be taken to be the identity. As a result this part of the 1-morphism does not affect the finality and so for the sake of readability we disregard it in the proof.
\end{Note}

Suppose that we have another element $Y\in S$, i.e. a cube 
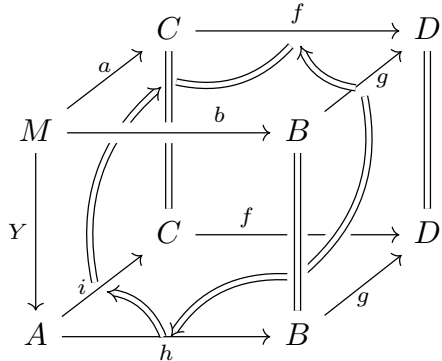
\[
\stik{1.2}{
{} \& C \ar{rr}[name=topback,below]{}[above]{f} \ar[equal]{dd}[name=backleft,sloped,xshift=-1.5pc]{} \& {} \& D \ar[equal]{dd}\\
M \ar[crossing over]{rr}[name=topfront,above]{}[above,near end,yshift=1pt]{b} \ar{dd}[name=frontleft]{}[left]{Y} \ar{ur}[name=topleft,below]{}[above]{a} \& {} \& B \ar[equal]{dd}[name=frontright,sloped,xshift=1pc]{} \ar{ur}[name=topright,below]{}[below,near end]{g} \& {}\\
{} \& C \ar{rr}[above,near start]{f} \& {} \& D\\
A \ar{rr}[name=bottomfront,below]{}[below]{h} \ar{ur}[name=bottomleft,sloped,yshift=-0.2pc]{}[above,near start]{i} \& {} \& B \ar{ur}[below]{g}
\arrow[Rightarrow,to path={(topright) to[bend left] (topback)}]{}{}
\arrow[Rightarrow,to path={(frontright) to[bend right] (bottomfront)}]{}{}
\arrow[equal,crossing over,to path={(frontright) to[bend right] (topright)}]{}{}
\arrow[Rightarrow,to path={(bottomfront) to[bend right] (bottomleft)}]{}{}
\arrow[Rightarrow,to path={(bottomleft) to[bend left] (backleft)}]{}{}
\arrow[equal,to path={(topback) to[bend left] (backleft)}]{}{}
\latearrow{/tikz/commutative diagrams/crossing over,/tikz/commutative diagrams/equal}{2-3}{4-3}{}
\latearrow{/tikz/commutative diagrams/crossing over}{2-1}{2-3}{}
}
\]

A morphism $Y\rightarrow C_R$ can be given by a commutative 4-cube
\[
\begin{tikzpicture}[commutative diagrams/every diagram,scale=2]
	 \tikzstyle{over}=[-,line width=4pt,draw=white];
	 \tikzstyle{CD}=[commutative diagrams/.cd, every arrow, every label]
	 \tikzstyle{equal}=[-,line width=1pt,double,double distance=3 pt]
	 \node (IBFL) at (0,0) {$A$};
	 \node (ITFL) at (0,1) {$M$};
	 \node (IBFR) at (1,0) {$B$};
	 \node (ITFR) at (1,1) {$B$};
	 \node (IBKL) at (0.23, 0.4) {$C$};
	 \node (ITKL) at (0.23,1.4) {$C$};
	 \node (IBKR) at (1.23,0.4) {$D$};
	 \node (ITKR) at (1.23,1.4) {$D$};
	 \node (OBFL) at (-1,-1) {$A$};
	 \node (OTFL) at (-1,2) {$M$};
	 \node (OTKL) at (-0.66,2.7) {$C$};
	 \node (OBKL) at (-0.66,-0.3) {$C$};
	 \node (OBFR) at (2,-1) {$B$};
	 \node (OBKR) at (2.34,-0.3) {$D$};
	 \node (OTFR) at (2,2) {$B$};
	 \node (OTKR) at (2.34,2.7) {$D$};
	 \path[CD]
        (OTKL) edge (OTKR)
        (OTKL) edge[equal] (OBKL)
        (OBKL) edge (OBKR)
        (OTKR) edge[equal] (OBKR)
        (OTKL) edge[equal] (ITKL)
        (OBKL) edge[equal] (IBKL)
        (OTKR) edge[equal] (ITKR)
        (OBKR) edge[equal] (IBKR)
        (ITKL) edge (ITKR)
        (ITKL) edge[equal] (IBKL)
        (IBKL) edge (IBKR)
        (ITKR) edge[equal] (IBKR)
        (ITFL) edge (ITKL)
        (IBFL) edge (IBKL)
        (ITFR) edge (ITKR)
        (IBFR) edge (IBKR)
        (ITFL) edge[over] (ITFR) edge (ITFR)
        (ITFL) edge[over] (IBFL) edge node[scale=1.5,left] {$C_R$} (IBFL)
        (IBFL) edge[over] (IBFR) edge (IBFR)
        (ITFR) edge[over] (IBFR) edge[equal] (IBFR)
        (OTFL) edge (OTKL)
        (OBFL) edge (OBKL)
        (OTFR) edge (OTKR)
        (OBFR) edge (OBKR)
        (ITFL) edge[over] (OTFL) edge[equal] (OTFL)
        (IBFL) edge[over] (OBFL) edge[equal] (OBFL)
        (ITFR) edge[over] (OTFR) edge[equal] (OTFR)
        (IBFR) edge[over] (OBFR) edge[equal] (OBFR)
        (OTFL) edge[over] (OTFR) edge (OTFR)
        (OTFL) edge[over] (OBFL) edge node[scale=1.5,left] {$Y$} (OBFL)
        (OBFL) edge[over] (OBFR) edge (OBFR)
        (OTFR) edge[over] (OBFR) edge[equal] (OBFR)
        ;
\end{tikzpicture}
\]
where the dimensions are ordered as $[\rightarrow,\downarrow$,radial towards center,$\nearrow$]. This defines the orientation of 2-morphisms and the order of their composition - see \autoref{app:commcube} for details. 
The diagram of 2-morphisms associated to the oriented 4-cube commutes iff each sub 3-cube commutes (see  \cite{graycoherence}). Hence, in all, the data we should provide is only one 2-face, namely the one shared by $Y$ and $C_R$, and this just amounts to a morphism $Y\xrightarrow{\psi} C_R$. We need to show that such a morphism which makes the 4-cube commute exists and is unique.

Consider the left 3-cube
\[
\stik{1}{
{} \& C \ar[equal]{rr}[name=topback,below]{}[above]{} \ar[equal]{dd}[name=backleft,sloped,xshift=-1.5pc]{}[right,yshift=-1.5pc]{} \& {} \& C \ar[equal]{dd}\\
M \ar[crossing over]{rr}[name=topfront,above]{}[below, xshift=1pc]{Y}\ar[equal]{dd}[name=frontleft]{}[left]{} \ar{ur}[name=topleft,below]{} \& {} \& A \ar[equal]{dd}[name=frontright,sloped,xshift=1pc]{} \ar{ur}[name=topright,below]{} \& {}\\
{} \& C \ar[equal]{rr} \& {} \&C\\
M \ar{rr}[name=bottomfront,below]{C_R} \ar{ur}[name=bottomleft,sloped,yshift=-0.2pc]{} \& {} \& A \ar{ur}
\arrow[Rightarrow,to path={(topright) to[bend left] (topback)}]{}{}
\arrow[Rightarrow,to path={(frontright) to[bend right] (bottomfront)},"\psi"']{}{}
\arrow[equal,crossing over,to path={(frontright) to[bend right] (topright)}]{}{}
\arrow[Rightarrow,to path={(bottomfront) to[bend right] (bottomleft)}]{}{}
\arrow[equal,to path={(bottomleft) to[bend left] (backleft)}]{}{}
\arrow[equal,to path={(topback) to[bend left] (backleft)}]{}{}
\latearrow{commutative diagrams/crossing over,commutative diagrams/equal}{2-3}{4-3}{}
\latearrow{commutative diagrams/crossing over}{2-1}{2-3}{}
\latearrow{commutative diagrams/crossing over}{2-3}{1-4}{}
}
\]
By assumption it commutes. Its bottom face is the left face of $C_R$; its top face is the left face of $Y$, denote it by $f_Y$; its front face is $\psi$; the rest are degenerate. So we get the equation
\[
\nstik{1}{
{} \ar{dd}{Y}  \ar[yshift=-1em,shorten <=1em,shorten >=0.2em]{dr}[xshift=-0em]{\psi}\& \ar{d}{a}  \& \ar{d}{a}\\
{} \ar[yshift=-1.5em,shorten <=1.2em]{ur} \& \ar{d}{i_R}  \ar[yshift=-1em,shorten <=1em,shorten >=0.2em]{dr}[xshift=0em]{\epsilon} \& \ar[equal]{d}{}\\
{} \ar{d}{i}  \& \ar{d}{i} \ar[yshift=-1.5em,shorten <=1.2em]{ur} \& \ar[equal]{d}{}\\
{} \& {} \& {}
}
=
f_Y
\]
So since $\epsilon$ is part of an adjunction, we have that
\begin{align*}
\psi & =
\nstik{1}{
{} \ar{dd}{Y} \ar[yshift=-1em,shorten <=1em,shorten >=0.2em]{dr}[xshift=-0em]{\psi} \& \ar{d}{a}  \& \ar{d}{a}  \& {} \ar{d}{a}\\
{} \ar[yshift=-1.5em,shorten <=1.2em]{ur} \& \ar{d}{i_R}  \&  {} \ar{d}{i_R} \ar[yshift=-1em,shorten <=1em,shorten >=0.2em]{dr}[xshift=0em]{\epsilon} \& {} \ar[equal]{dd}\\
{} \ar[equal]{dd}{}  \& \ar[yshift=-1em,shorten <=1em,shorten >=0.2em]{dr}[xshift=-0em]{\eta} \ar[equal]{dd}{} \&  \ar{d}{i}  \ar[yshift=-1.5em,shorten <=1.2em]{ur}\& {} \\
{}  \& {} \ar[yshift=-1.5em,shorten <=1.2em]{ur} \& {}  \ar{d}{i_R}\& {} \ar{d}{i_R}\\
{}  \& {} \& {} \& {}
}
\\
& =
\nstik{1}{
{} \ar{dd}{Y}  \& {} \ar[yshift=-1em,shorten <=1em,shorten >=0.2em]{dr}[xshift=-0em]{\psi} \ar{dd}{Y}  \& \ar{d}{a}  \& {} \ar{d}{a}\\
{} \& {} \ar[yshift=-1.5em,shorten <=1.2em]{ur}  \&  {} \ar{d}{i_R} \ar[yshift=-1em,shorten <=1em,shorten >=0.2em]{dr}[xshift=0em]{\epsilon} \& {} \ar[equal]{dd}\\
{} \ar[equal]{dd}{}\ar[yshift=-1em,shorten <=1em,shorten >=0.2em]{dr}[xshift=-0em]{\eta}  \&  \ar{d}{i} \&  \ar{d}{i}  \ar[yshift=-1.5em,shorten <=1.2em]{ur}\& {} \\
{} \ar[yshift=-1.5em,shorten <=1.2em]{ur} \& {}\ar{d}{i_R}  \& {}  \ar{d}{i_R}\& {} \ar{d}{i_R}\\
{}  \& {} \& {} \& {}
}
=
 f_Y \circ \eta
\end{align*}
where the second step is justified by the four-interchange law in a bicategory. 

So $\psi$ is uniquely defined, and all we need to check is that it makes the front cube in the 4-cube commute.

The front cube is 
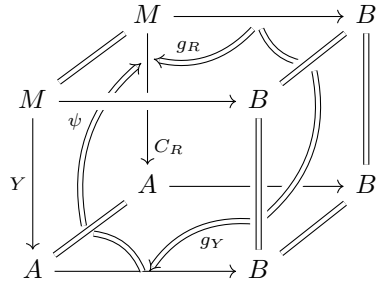
\[
\stik{1}{
{} \& M \ar{rr}[name=topback,below]{}[above]{} \ar{dd}[name=backleft,sloped,xshift=-1.5pc]{}[right,yshift=-1.5pc]{C_R} \& {} \& B \ar[equal]{dd}\\
M \ar[crossing over]{rr}[name=topfront,above]{}[below, xshift=1pc]{} \ar{dd}[name=frontleft]{}[left]{Y} \ar[equal]{ur}[name=topleft,below]{} \& {} \& B \ar[equal]{dd}[name=frontright,sloped,xshift=1pc]{} \ar[equal]{ur}[name=topright,below]{} \& {}\\
{} \& A \ar{rr} \& {} \&B\\
A \ar{rr}[name=bottomfront,below]{} \ar[equal]{ur}[name=bottomleft,sloped,yshift=-0.2pc]{} \& {} \& B \ar[equal]{ur}
\arrow[equal,to path={(topright) to[bend left] (topback)}]{}{}
\arrow[Rightarrow,to path={(frontright) to[bend right] (bottomfront)},"g_Y"]{}{}
\arrow[equal,crossing over,to path={(frontright) to[bend right] (topright)}]{}{}
\arrow[equal,to path={(bottomfront) to[bend right] (bottomleft)}]{}{}
\arrow[Rightarrow,to path={(bottomleft) to[bend left] (backleft)},"\psi"]{}{}
\arrow[Rightarrow,to path={(topback) to[bend left] (backleft)},"g_R"']{}{}
\latearrow{commutative diagrams/crossing over,commutative diagrams/equal}{2-3}{4-3}{}
\latearrow{commutative diagrams/crossing over}{2-1}{2-3}{}
\latearrow{commutative diagrams/crossing over,commutative diagrams/equal}{2-3}{1-4}{}
}
\]
So we need to show that $\psi\circ g_Y=g_R$.

Recall that
\[
g_R=\nstik{1}{
{} \ar{d}{b} \& \ar{d}{b} \ar[yshift=-1em,shorten <=1em,shorten >=0.2em]{dr}[xshift=0em]{\beta} \& \ar{d}{a} \& \ar{d}{a}\\
{} \ar[equal]{d} \ar[yshift=-1em,shorten <=1em,shorten >=0.2em]{dr}[xshift=-0.5em]{\eta} \& \ar{d}{g} \ar[yshift=-1.5em,shorten <=1.2em]{ur} \& \ar{d}{f} \ar[yshift=-1em,shorten <=1em,shorten >=0.2em]{dr}[xshift=0em]{\alpha_R^{-1}} \& \ar{d}{i_R}\\
{} \ar[equal]{d} \ar[yshift=-1.5em,shorten <=1.2em]{ur} \& \ar{d}{g_R} \& \ar{d}{g_R} \ar[yshift=-1.5em,shorten <=1.2em]{ur} \& \ar{d}{h}\\
{} \& {} \& {} \& {}
}
\]
Denoting the front face of the outer cube (the "$Y$"-cube ) by $g_Y$, we have $\beta=f_Y\circ\alpha\circ g_Y$.

Substituting we get
\[g_R=
\nstik{1}{
{} \ar{dd}{b} \& \ar{dd}{b} \ar[yshift=-1em,shorten <=1em,shorten >=0.2em]{dr}[xshift=0em]{g_Y}\& \ar{d}{Y} \& \ar{d}{Y} \ar[yshift=-1em,shorten <=1em,shorten >=0.2em]{dr}[xshift=0em]{f_Y}\& \ar{dd}{a} \& \ar{dd}{a}\\
{} \& {} \ar[yshift=-1.5em,shorten <=1.2em]{ur} \& \ar{d}{h} \ar[yshift=-1em,shorten <=1em,shorten >=0.2em]{dr}[xshift=-0.5em]{\alpha}\& \ar{d}{i} \ar[yshift=-1.5em,shorten <=1.2em]{ur}\& {}\& {}\\
{} \ar[equal]{d} \ar[yshift=-1em,shorten <=1em,shorten >=0.2em]{dr}[xshift=-0.5em]{\eta} \& \ar{d}{g}  \& \ar{d}{g} \ar[yshift=-1.5em,shorten <=1.2em]{ur}  \& \ar{d}{f} \& {} \ar{d}{f} \ar[yshift=-1em,shorten <=1em,shorten >=0.2em]{dr}[xshift=0em]{\alpha_R^{-1}} \& {} \ar{d}{i_R}\\
{} \ar[equal]{d} \ar[yshift=-1.5em,shorten <=1.2em]{ur} \& \ar{d}{g_R} \& \ar{d}{g_R}  \& {} \ar{d}{g_r} \& {} \ar{d}{g_R} \ar[yshift=-1.5em,shorten <=1.2em]{ur} \& {} \ar{d}{h}\\
{} \& {} \& {} \& {} \& {} \& {}
}
\]

Now we plug in the identity of $i$ in the fourth column as a composition of unit/counit maps, and interchange $g_Y$ and $\eta$:
\[g_R=
\nstik{0.95}{
{} \& {} \& {} \& {}\ar{d}{Y} \& {} \ar{d}{Y} \& {} \ar{d}{Y} \ar[yshift=-1em,shorten <=1em,shorten >=0.2em]{dr}[xshift=0em]{f_Y}\& {}\ar{dd}{a} \& {}\ar{dd}{a}\\
{} \& {} \& {} \& {}\ar[yshift=-1em,shorten <=1em,shorten >=0.2em]{dr}[xshift=0em]{\eta'}\ar[equal]{dd} \& {} \ar{d}{i} \& {} \ar{d}{i} \ar[yshift=-1.5em,shorten <=1.2em]{ur} \& {} \& {}\\
{} \ar{dd}{b} \ar[yshift=-1em,shorten <=1em,shorten >=0.2em]{dr}[xshift=0em]{g_Y} \& \ar{d}{Y} \& \ar{d}{Y} \& {} \ar[yshift=-1.5em,shorten <=1.2em]{ur} \& {} \ar{d}{i_R} \ar[yshift=-1em,shorten <=1em,shorten >=0.2em]{dr}[xshift=0em]{\epsilon'}\& {} \ar[equal]{dd} \& {}  \ar[equal]{dd} \& \ar[equal]{dd}\\
{} \ar[yshift=-1.5em,shorten <=1.2em]{ur} \& {}\ar{d}{h} \& {}\ar{d}{h} \ar[yshift=-1em,shorten <=1em,shorten >=0.2em]{dr}[xshift=-0.5em]{\alpha}\& {}\ar{d}{i} \& {} \ar{d}{i} \ar[yshift=-1.5em,shorten <=1.2em]{ur} \& {}\& {}\& {}\\
{} \ar[equal]{dd}  \& {} \ar[equal]{dd} \ar[yshift=-1em,shorten <=1em,shorten >=0.2em]{dr}[xshift=-0.5em]{\eta} \& \ar{d}{g} \ar[yshift=-1.5em,shorten <=1.2em]{ur}  \& \ar{d}{f} \& {} \ar{d}{f}  \& {} \ar{d}{f} \& {}\ar{d}{f}  \ar[yshift=-1em,shorten <=1em,shorten >=0.2em]{dr}[xshift=0em]{\alpha_R^{-1}}\& {}\ar{d}{i_R} \\
{}  \& \ar[yshift=-1.5em,shorten <=1.2em]{ur} \& \ar{d}{g_R}  \& {} \ar{d}{g_r} \& {} \ar{d}{g_R}  \& {} \ar{d}{g_R}\& {}\ar{d}{g_R} \ar[yshift=-1.5em,shorten <=1.2em]{ur}\& {} \ar{d}{h}\\
{} \& {} \& {} \& {} \& {} \& {}\& {} \& {}
}
\]

using some interchanging we get to
\[g_R=
\nstik{0.95}{
{} \& {}\ar{d}{Y} \& {}\ar{d}{Y} \& {}\ar{d}{Y} \& {} \ar{d}{Y} \& {} \ar{d}{Y} \& {}\ar{d}{Y} \ar[yshift=-1em,shorten <=1em,shorten >=0.2em]{dr}[xshift=0em]{f_Y} \& {}\ar{dd}{a}\\
{} \& {}\ar[equal]{dd} \ar[yshift=-1em,shorten <=1em,shorten >=0.2em]{dr}[xshift=0em]{\eta'} \& {} \ar{d}{i} \& {} \ar{d}{i} \& {} \ar{d}{i} \& {} \ar{d}{i}  \& {}\ar{d}{i}\ar[yshift=-1.5em,shorten <=1.2em]{ur} \& {}\\
{} \ar{dd}{b} \ar[yshift=-1em,shorten <=0em,shorten >=0.2em]{dr}[xshift=0em]{g_Y} \& {} \ar[yshift=-1.5em,shorten <=1.2em]{ur} \& {} \ar{d}{i_R} \& {}  \ar{d}{i_R} \& {} \ar{d}{i_R} \ar[yshift=-1em,shorten <=1em,shorten >=0.2em]{dr}[xshift=0em]{\epsilon'}\& {} \ar[equal]{dd} \& {}  \ar[equal]{dd} \& \ar[equal]{dd}\\
{} \ar[yshift=-1.5em,shorten <=1.2em]{uuur} \& {}\ar{d}{h} \& {}\ar{d}{h} \& {} \ar[yshift=-1em,shorten <=1em,shorten >=0.2em]{dr}[xshift=-0.5em]{\alpha} \ar{d}{h} \& {} \ar{d}{i} \ar[yshift=-1.5em,shorten <=1.2em]{ur} \& {}\& {}\& {}\\
{} \ar[equal]{dd}  \& {} \ar[equal]{dd}  \& {} \ar[equal]{dd} \ar[yshift=-1em,shorten <=1em,shorten >=0.2em]{dr}[xshift=-0.5em]{\eta}   \& \ar{d}{g} \ar[yshift=-1.5em,shorten <=1.2em]{ur} \& {} \ar{d}{f}  \& {} \ar{d}{f}\ar[yshift=-1em,shorten <=1em,shorten >=0.2em]{dr}[xshift=0em]{\alpha_R^{-1}} \& {}\ar{d}{i_R}  \& {}\ar{d}{i_R} \\
{}   \& {} \& {} \ar[yshift=-1.5em,shorten <=1.2em]{ur} \& {} \ar{d}{g_R} \& {} \ar{d}{g_R}  \& {} \ar{d}{g_R} \ar[yshift=-1.5em,shorten <=1.2em]{ur} \& {}\ar{d}{h} \& {} \ar{d}{h}\\
{} \& {} \& {} \& {} \& {} \& {}\& {} \& {}
}
\]

Finally, we can cancel the middle part, which is just $\alpha_R^{-1}\circ\alpha_R$ and get \[
g_R=f_Y\circ \eta'\circ g_Y=\psi\circ g_Y
\]
as required. This completes the proof.
\end{proof}

\begin{Remark}
A completely dual argument yields an initial object in the Hom-category, which is not isomorphic to the final object, so it is not the case that the Hom-category is contractible. 
\end{Remark}

\section{The category \texorpdfstring{$\PP$}{P} of polynomial functors}
\label{sec:p}
\subsection{Recollection of \texorpdfstring{$\PP$}{P} and its Grothendieck group}
\label{ssec:defP}
We consider the category $\PP$ of polynomial functors over a field $\kk$ of characteristic 0, defined by Friedlander and Suslin in \cite{FrieSusPoly}. 
\begin{Definition}
The category of polynomial functors $\PP$ is the category whose objects are functors from $\Vect$ to $\Vect$ that induce polynomial maps on the $\Hom$ spaces, i.e. \[F:\Vect\rightarrow\Vect\] such that for any two spaces $V,W$, the map \[\Hom_\kk(V,W)\rightarrow\Hom_\kk(FV,FW)\] is a polynomial map.
\end{Definition}
\begin{Definition}
For any finite set $S$ we consider the category $\PPS{S}$ of polynomial functors from the category $\Sh(S)$ of sheaves of vector spaces over $S$ to $\Vect$.
\end{Definition}

It is easy to check that $\PPS{S}$ satisfies the universal property of the tensor product of categories in $\CatAd$ as defined by Deligne in \cite{deligne}. Namely, for any set $S$ we have the functor 
\begin{align}
\label{tensorstructuremaps}
\boxtimes_S:\PP^{\times S} &\rightarrow \PPS{S} & \forall V \in \Sh(S) (\boxtimes_S F_s)(V_s)&=\otimes_S F_s(V_s)
\end{align}
that presents $\PPS{S}$ as the Deligne tensor of $\PP^{\times S}$.

\begin{Note}
The Grothendieck $K$-group of $\PP$ has the structure of an algebra coming from the tensor product on $\PP$. See also \autoref{SSH-kgroup} for the description of the algebra structure in terms of the Hopf category structure on $\PP$.
\end{Note}

\begin{Proposition} 
\label{prop:PolyGrothGroup}
The Grothendieck $K$-group of $\PP$ is isomorphic as an algebra to $\Lambda$ - the algebra of symmetric polynomials in a countable number of variables.
\end{Proposition}

\begin{proof}
As is shown in \cite{FrieSusPoly}, the subcategory of polynomial functors of degree $\leq d$ is equivalent to the category of polynomial representations of $\GL_n$ of degree $\leq d$, when $n\geq d$. Sending a representation to its character on the torus gives an morphism of the $K$ group of this subcategory with the $\ZZ$-group of symmetric polynomials in $n$ variables of degree $\leq d$. Going to the limit gives us a morphism $F:K(\PP)\rightarrow \Lambda$.

This morphism is multiplicative since the character of the tensor product of representations is the product of the characters. It is injective because a rational representation of $\GL_n$ is determined up to isomorphism by its character on the torus. It is surjective since the images of the polynomial functor $V\mapsto\Sym^nV$ generate $\Lambda$ (they are the \emph{whole} symmetric functions).
\end{proof}
\begin{Corollary}
The Grothendieck $K$-group of $\PPS{S}$ is isomorphic to $\Lambda^{\otimes S}$.
\end{Corollary}
\begin{Remark}
It is shown in \cite{Zelbook} that $\Lambda \cong \bigoplus_n K(\Rep(S_n))$ 
\end{Remark}

\subsection{The SSH structure on \texorpdfstring{$\PP$}{P}}
\label{sec:selfadjointhopf}
In this section we put a SSH structure on $\PP$, in the sense of \autoref{sec:SSHDef}, i.e. we construct a symmetric monoidal functor $\FinDisj\rightarrow\CatAdTen$ preserving comma squares which sends the set $\fset{1}$ to the category $\PP$.

\begin{Proposition}
The following collection of data gives a SSH structure on $\PP$.
\begin{itemize}
\item To any finite set $S$, we assign the category $\PPS{S}$ defined above.
\item To any map of finite sets $\varphi:S\rightarrow T$ we assign the functor \[m_\varphi:\PPS{S}\rightarrow\PPS{T}\] defined by the formula
\[
m_\varphi(F):=F\circ \varphi^*
\]
note that it has a natural adjoint (both left and right),$\Delta_\varphi$, given by \[
\Delta_\varphi(\Phi):=\Phi\circ \varphi_*
\]
and hence is in $\CatAd$

\item To any map $(S_i)\xrightarrow{(\varphi_i)} T$ in $\FinDisj$, over the active map $\pset{n}\rightarrow\pset{1}$, we assign the functor
\[
m_{(\varphi_i)}:\prod \PPS{S_i}\rightarrow\PPS{T}
\]
given by 
\[
m_{(\varphi_i)}(F_i)(V)=\bigotimes F_i(\varphi_i^* V)
\]
\item To any commutative square of sets we associate the $2$-commutative diagram
\[
\tik
S\ar{r}{a} \ar{d}{c} &T \ar{d}{b} \\ 
R \ar{r}{d} & U 
\tak
\rightsquigarrow
\tik
\PPS{S}\arrow[]{d}{m_c} \arrow[]{r}{m_a} & \PPS{T} \arrow[]{d}{m_b} \arrow[shorten >=0.4cm,shorten <=0.4cm,Rightarrow]{dl}[above,sloped]{\sim} \\ 
\PPS{R} \arrow[]{r}{m_d} & \PPS{U}
\tak
\]
with an isomorphism in the middle which comes from the isomorphims $\varphi^*\psi^*\cong (\psi\varphi)^*$.
\end{itemize}

\end{Proposition}

\begin{proof}
Denote by $\Hop$ the resulting functor. As noted in the proof of \autoref{prop:PolyGrothGroup} $\PP$ is equivalent to a direct sum of subcategories of finite dimensional representations of general linear groups. Explicitly: \[
\PP \cong \bigoplus_d \Rep_{\deg=d}(\GL_d)
\]
So all categories involved in the construction are graded 2-vector spaces.

First, we have explicit adjunctions $(m_\varphi\dashv\Delta_\varphi)$ and $(\Delta_\varphi\dashv m_\varphi)$ given by the explicit adjunctions $(\varphi_*\dashv\varphi^*)$ and $(\varphi^*\dashv\varphi_*)$(see \autoref{app:pullpushadjsets}). So $\Hop$ lands in $\CatAd$. That it is symmetric monoidal follows from \eqref{tensorstructuremaps}.

Regarding the other conditions:

Obviously, $\Hop(\emptyset)$ is canonically equivalent to $\Vect$, and $\Hop(\fset{1})$ is canonically equivalent to $\PP$.

It remains to show that $\Hop$ sends Cartesian squares in the fiber over $\pset{1}$ to squares satisfying the BC condition.
Consider a Cartesian square of maps of sets and the corresponding diagram of tensor powers 
\[
\tik
S\ar{r}{a} \ar{d}{c} &T \ar{d}{b} \\ 
R \ar{r}{d} & U 
\tak
\rightsquigarrow
\tik
\PPS{S}\arrow[]{d}{m_c} \arrow[]{r}{m_a} & \PPS{T} \arrow[]{d}{m_b} \arrow[shorten >=0.4cm,shorten <=0.4cm,Rightarrow]{dl}[above,sloped]{\alpha} \\ 
\PPS{R} \arrow[]{r}{m_d} & \PPS{U}
\tak
\]

Since the funtors $m_\varphi$ were defined as precomposition with $\varphi^*$, it is enough to show that the square 
\[
\stik{1}{
\Sh(S) \& \Sh(T) \ar{l}{a^*} \arrow[shorten >=0.4cm,shorten <=0.4cm,Rightarrow]{dl}[above,sloped]{\sim} \\
\Sh(R) \ar{u}{c^*} \& \Sh(U) \ar{l}{d^*} \ar{u}{b^*}
}
\]
satisfies the Beck-Chevalley condition, but this follows immediately from proper base change, so we are done.

\end{proof}


\section{The Heisenberg double associated to positive self-adjoint Hopf algebras}
\label{sec:algebraFock}


\subsection{The Heisenberg double}
\label{ssec:algebraFock}

Given a pair of Hopf algebras $H,H'$ and a pairing between them satisfying certain requirements one can endow the space $H\otimes H'$ with an algebra structure and form an algebra called the \emph{Heisenberg double}. The description of the general notion can be found in \cite{STSMHeisdouble} and \cite{kapranovheisenberg}. In this article we are concerned with the specific case of this construction in the case of positive selfadjoint Hopf algebras introduced by Zelevinsky in \cite{Zelbook}. Let us recall the definitions:
\begin{Definition}
A positive selfadjoint Hopf (PSH) algebra is a graded connected Hopf algebra over $\ZZ$ with an inner product and a distinguished finite orthogonal $\ZZ$ basis in each grade s.t. multiplication and comultiplication are adjoint and take elements with positive coefficients to elements with positive coefficients.
\end{Definition}

\begin{Definition}
A graded Hopf algebra over $\ZZ$ $A=\bigoplus_{n\geq 0}A_n$ is called \emph{connected} if the unit morphism and the counit morphism restricted to $A_0$ give an isomorphism of $A_0$ with $\ZZ$
\end{Definition}

For a positive self-adjoint Hopf algebra $A$ let us consider the dual pair of Hopf algebras $(A,A)$ and outline the construction of the Heisenberg double in this case. Simultaneously we will show that this algebra has a natural action on $A$ (considered as a $\ZZ$ module). We will refer to this action as \emph{Fock space action} by the analogy with the case of infinite dimensional Heisenberg algebra, which is a special case of this construction as noted in \autoref{ssec:generators}. Applying the analog of this construction in the categorical setting will allow us to construct the natural categorical Fock space action for any SSH category in \autoref{sec:heisenbergaction}.

Consider a PSH algebra $A$ and denote the adjoint multiplication and comultiplication maps by
\begin{align*}
m: A\otimes A &\rightarrow A &\Delta:A&\rightarrow A\otimes A
\end{align*}
Note that $A$ is commutative and cocommutative, a result proven in \cite{Zelbook}.

For each $x\in A$ we define operators $m_x,\Delta_x:A\rightarrow A$ by the formulas
\begin{equation}
\label{def:mxdeltax}
\begin{aligned}
m_{x} &=m\circ i_x & \Delta_x &=j_x\circ\Delta
\end{aligned}
\end{equation}
where 
\begin{align*}
i_x(y)&=x\otimes y  & j_x(y\otimes z)&=z<x,y>
\end{align*}
\begin{Remark}
Note that $m_x, \Delta_x$ are adjoint for any $x \in A$. 
\end{Remark}
We use these operators to define a morphism of $\ZZ$ groups 
\begin{align*}
\varphi:A\otimes A &\rightarrow\End_\ZZ(A) &x\otimes y &\mapsto m_x\Delta_y
\end{align*}
\begin{Proposition}
\label{prop:heisdouble}
$\varphi$ is injective and its image is a subalgebra of $End_{\ZZ}(A)$
\end{Proposition}
\begin{Notation}
Since $\varphi$ is injective, it induces an algebra structure on $A\otimes A$. We denote the algebra $A\otimes A$, with the algebra structure given by $\varphi$, by $\Heis(A)$. The action of $\Heis(A)$ on $A$ defined by $\varphi$ is called the \emph{Fock space action}.
\end{Notation}
\begin{proof}
To prove that the image of $\varphi$ is a subalgebra note the we have the following relations $\forall x,y \in A$:
\begin{align}
\label{ref:heisenbergrelations}
m_xm_y &= m_{m(xy)}=m_{m(yx)}=m_ym_x\\
\Delta_x\Delta_y &= \Delta_{yx}=\Delta_{xy}=\Delta_y\Delta_x\\
\label{eqn:deltam}
\Delta_xm_y &= m\Delta^2_{\Delta(x)}i_y
\end{align}

where if $\Delta(x)=x_{(1)}\otimes x_{(2)}$ (in Sweedler notation) then 
\begin{equation}
\Delta^2_{\Delta(x)}i_y(z)=\Delta^2_{\Delta(x)}(y\otimes z):=\Delta_{x_{(1)}}y\otimes\Delta_{x_{(2)}}z
\end{equation}

The first two relations hold since multiplication and comultiplication in $A$ are associative and commutative and the relation \eqref{eqn:deltam} holds since (as shown in \cite{Zelbook}) $\forall z,u \in A$:
\begin{multline*}
<\Delta_xm_yz,u>=<\Delta_xm(y\otimes z),u>=\\
<m(y \otimes z), m(x \otimes u)>=<y\otimes z, \Delta m(x \otimes u)>=\\
<y \otimes z,m(\Delta x \otimes \Delta u)>=<y \otimes z,m((x_{(1)}\otimes x_{(2)})\otimes\Delta u)>=\\
<\Delta_{x_{(1)}}y \otimes\Delta_{x_{(2)}}z,\Delta u>=<m(\Delta_{x_{(1)}}y \otimes\Delta_{x_{(2)}}z),u>
\end{multline*}
(we used the fact that $m_{x_{(1)}\otimes x_{(2)}}$ is adjoint to $\Delta_{x_{(1)}\otimes x_{(2)}}$ on $A^{\otimes 2}$)\\
So explicitly, the third relation gives us that 
\[
\Delta_xm_y=m_{\Delta_{x_{(2)}}(y)}\Delta_{x_{(1)}}\in\varphi(A\otimes A) 
\]

To prove that $\varphi$ is injective we use the fact that $A$ is graded and each grade has an orthogonal basis. Let $\sum_i x_iy_i \in Ker \varphi$ and assume without loss of generality that $y_i$ are the elements of the orthogonal basis of $A$. Let $y_{i_0}$ be an element of minimal degree so that $x_{i_0} \neq 0$ and let $r=\deg y_{i_0}$. It follows from the definition of $\Delta_y$ that for every $z \in A_n, y \in A_m$ $\Delta_yz=0$ when $n<m$, and $\Delta_yz=<y,z>$ for $n=m$. Hence by applying the operator $\sum_i m_{x_i}\Delta_{y_i}$ to $y_{i_0}$ we get $x_{i_0}=0$, a contradiction.
\end{proof}
To construct the categorification of the Fock space action in \autoref{ssec:heisenbergcategorification} we will use the following key
\begin{Observation}
\label{heisenbergrelation}
The relation $\Delta_xm_y = m\Delta^2_{\Delta(x)}i_y$ for any $x,y$ is equivalent to the relation $\Delta_xm = m\Delta^2_{\Delta(x)}$ for any $x$.
\end{Observation}
We will call this relation the \emph{Heisenberg relation} throughout this article. The reformulation above a first step towards the construction of the categorification of Heisenberg algebra from the SSH structure.

The above construction can be easily generalized to give the following general statement:
\begin{Proposition}
\label{prop:heisenbergaction}
Giving an action of $\Heis(A)$ on a space $V$ is the same as giving a morphism of spaces $A\otimes A\xrightarrow{a}\End(V)$ which satisfies
\begin{enumerate} 
\item The restrictions of $a$ to $A\otimes 1$ and $1\otimes A$ are morphisms of algebras.
\item  Denote \begin{gather*}
\Delta^2_a:A\otimes A\rightarrow\End(A\otimes V) \\
\Delta^2_a(x\otimes y)(z\otimes v):=\Delta_x(z)\otimes (a(1\otimes y)(v))\\
m_a:A\otimes V\rightarrow V \\
m_a(z\otimes v):=a(z\otimes 1)(v)
\end{gather*}
then for any $x\in A$ we have $a(1\otimes x)\circ m_a=m_a\circ \Delta^2_a(\Delta(x))$
\end{enumerate}
\end{Proposition}


\subsection{The infinite-dimensional Heisenberg algebra}
\label{ssec:generators}
The classical one-variable Heisenberg algebra is the $\ZZ$-algebra with two generators $p,q$ and one defining relation $[p,q]=1$. The infinite version of this algebra is usually defined in terms of infinite number of generators and relations. There are several different versions used in different settings, some of which are described below. All of these are (sometimes non-isomorphic) $\ZZ$ forms of the same complex algebra.

We use the notion of of Heisenberg double for a positive self-adjoint Hopf algebra which we explored in the previous section to give an alternative description of the infinite-dimensional Heisenberg algebra. This description doesn't use the language of generators and relations, and thus lends itself more naturally to categorification.

Denote by $\Lambda$ the algebra of symmetric polynomials in a countable number of variables over $\ZZ$. $\Lambda$ has the structure of a PSH algebra. The $\ZZ$ basis is given by the Schur polynomials; this also defines an inner product. The multiplication map $m$ is given by the multiplication of polynomials.

\begin{Definition}
We define the Heisenberg algebra of infinite rank to be $\Heis(\Lambda)$, i.e. the the Heisenberg double corresponding to the pair $(\Lambda,\Lambda)$.
\end{Definition}

Let us describe how some of the commonly used definitions of Heisenberg algebra arise from the above definition:
\begin{enumerate}
\item A $\ZZ$-algebra with generators $p_n,q_n,n\in\NN$ and relations\begin{itemize}
\item 
$[p_m,p_n]=[q_m,q_n]=0$
\item
$[p_m,q_n]=\delta_{mn}1$
\end{itemize} 
This corresponds to taking $p_n$ to be the elementary symmetric function of degree $n$ ($\sum_{i_1<\dots<i_n} x_{i_1}\cdots x_{i_n}$) in the left $\Lambda$ and the $q_n$ to be the primitive symmetric function of degree $n$ (described in \cite{Zelbook}) in the right $\Lambda$.
\item A $\ZZ$-algebra with generators $c_k,k\in\ZZ$ and relations $[c_k,c_l]=k\delta_{k+l,0}$. This corresponds to taking $c_k$ with positive $k$ to be the primitive symmetric function
of degree $k$ in the right $\Lambda$ and $c_k$ with negative $k$ to be the primitive symmetric
function of degree $k$ in the left $\Lambda$, and taking $c_0=1$. 
\item (The algebra which Khovanov categorifies in \cite{Khov})  Generators $a_n,b_n$, $n\in\NN$ with relations\begin{itemize}
\item
$a_0=b_0=1$
\item 
$[a_m,a_n]=[b_m,b_n]=0$
\item
$[a_m,b_n]=b_{n-1}a_{m-1}$
\end{itemize} 
This corresponds to taking $a_n$ to be the elementary symmetric function
of degree $n$ in the left $\Lambda$ and the $b_n$ to be the whole symmetric
function of degree $n$ ($\sum_{i_1\leq\dots\leq i_n} x_{i_1}\cdots x_{i_n}$) in the right $\Lambda$.
\end{enumerate}

\section{A categorical Heisenberg double}
\label{sec:CategoricalHeisenberg}
\label{sec:heisenbergaction}
\label{ssec:heisenbergcategorification}
\subsection{The statement}
Our starting point in this section is the attempt to construct a categorification of the Fock space representation of the Heisenberg double discussed in \autoref{sec:algebraFock}. There we described how this representation is constructed using the PSH algebra structure on $A$. Presently from a self-adjoint Hopf category structure on a category $\CCC$ we want to construct a categorical action on $\CCC$ which descends to the Fock space representation of $\Heis(K(\CCC))$.

Denote the left and right adjoints of the functor $m:\CS{\fset{2}}\rightarrow \CCC$ by $\Delta^l$ and $\Delta^r$. 
For the straightforward categorification of \autoref{prop:heisenbergaction} in the case of the Fock space action we need to construct the following for the SSH category $\CCC$:\\
\begin{enumerate}
\label{list:categoricalFock}
\item Endofunctors $m_F,\Delta^r_F$ of $\CCC$, for any $F\in\CCC$, and endofunctors $(\Delta^r)^2_\Phi$ of $\CS{[2]}$ for any $\Phi\in\CS{[2]}$.
\item Isomorphisms $m_Fm_G\cong m_{F\otimes G}$ and $\Delta^r_G\Delta^r_F\cong \Delta^r_{F\otimes G}$.
\item Isomorphisms $\Delta^r_F m\cong m {(\Delta^r)^2}_{\Delta^l(F)}$
\end{enumerate}

\begin{Remark}
\label{rem:leftright}
The construction of the isomorphism $3$ as outlined in the proof of Theorem \ref{th:deltam} below forces us to use both the left and right adjoint of the multiplication, although the roles may be reversed. Briefly, this happens because we use the unit of \emph{left} adjunction to construct a certain canonical square of functors which then has a commutative \emph{right} mate. This is evocative of the construction in \cite{savageyacobi} where the authors construct lifts of the Heisenberg relations in several examples where the left and right adjoints are not isomorphic.
\end{Remark}

In the view of the \autoref{prop:heisenbergaction} this would give us

\begin{Theorem}
\label{Th:HeisCat}
Let $\CCC$ be an SSH category and denote by $\EndAd(\CCC)$ the category of endofunctors of $\CCC$ admitting adjoints.\\
The functor $\CCC\otimes\CCC^{op}\rightarrow \EndAd(\CCC)$ given by $F\boxtimes G\mapsto m_F\circ \Delta^r_G$ can be naturally constructed using the SSH structure on $\CCC$ and descends to the Fock space representation of the Heisenberg double $\Heis(K(\CCC))$.
\end{Theorem}
In the following section we will use the SSH structure on $\CCC$ to construct the adjoint functors $i_F,j^r_F,j^l_F$ and define $m_{F}=m\circ i_F, \Delta^r_F=j^r_F\circ\Delta^r, \Delta^l_F=j^l_F\circ\Delta^l$, and similarly for the extensions to $\CS{[2]}$. We will then show that the $2$-morphisms in question are constructed naturally from the SSH structure. Their invertibility is a continuity statement (see \autoref{prop:bccart}), which follows from the continuity of the SSH functor (or, equivalently, from the fact that certain squares satisfy the Beck-Chevalley condition).

In \autoref{ssec:FockP} we will apply our general construction to the case of the SSH category $\PP$ of polynomial functors and explicitly describe the $1$- and $2$-morphisms given by it in this example. 


\subsection{Constructions}
\label{ssec:CatHeisenbergConstructions}
Consider the functor $\boxtimes:\CCC^{\times 2}\rightarrow\CS{\fset{2}}$ that is the image under $\Hop$ of the coCartesian arrow $(\{1\},\{2\}) \cartesianarrow \{1,2\}$. We define for any object $F\in\CCC$, the functor $i_F$ as
\begin{align*}
i_F:\CCC & \rightarrow\CS{\fset{2}}, & i_F(X)=F\boxtimes X
\end{align*}
Note that by our requirements $i_F$ has left and right adjoints which we denote by $j^l_F,j^r_F$
\[
j^l_F,j^r_F: \CS{\fset{2}}\rightarrow \CCC
\]
We define $m_F,\Delta^l_F,\Delta^r_F$ by
\begin{align*}
m_{F}&=m\circ i_F & \Delta^r_F&=j^r_F\circ\Delta^r & \Delta^l_F&=j^l_F\circ\Delta^l 
\end{align*} 
$\Delta^r_F, \Delta^l_F$ are respectively right and left adjoint to $m_F$ for any $F \in \CCC$. 

The construction of first two isomorphisms in \eqref{list:categoricalFock} is fairly straightforward. To construct the isomorphism $m_Fm_G\cong m_{F\otimes G}$ we first note that the functor $m_{F\otimes G}$ is canonically isomorphic to the functor $m_{m_F G}$ - the isomorphism $F\otimes G \cong m_F G$ is a part of the self-adjoint Hopf structure on $\CCC$ as explained in \autoref{sec:SSHDef}. The functors $m_Fm_G$ and $m_{m_F G}$ are the images of the following morphisms in $\FinDisj$ under $\Hop$:

\begin{align*}
m_Fm_GH: & {} & \nnstik{1}{
(F) \fset{1} \ar{r} \& \fset{1} \ar{r} \& \fset{1} \ar{dr} \& \\
(G) \fset{1} \ar{dr} \& \& \& \fset{2}\ar{r} \& \fset{1} \\
\& \fset{2} \ar{r} \& \fset{1} \ar{ur} \\
(H) \fset{1} \ar{ur}
}\\
m_{m_F G}H: & {} & \nnstik{1}{
(F) \fset{1} \ar{dr} \\
\& \fset{2} \ar{r} \& \fset{1} \ar{dr} \\
(G) \fset{1} \ar{ur} \& \& \& \fset{2}\ar{r} \& \fset{1} \\
(H) \fset{1} \ar{r} \& \fset{1} \ar{r} \& \fset{1} \ar{ur}
}
\end{align*}
These morphisms form a commutative diagram in $\FinDisj$, hence we have an isomorphism $m_Fm_G \cong m_{m_F G}$ whenever we have a self-adjoint Hopf structure on $\CCC$.

Passing to adjoints we get an isomorphism
\[
\Delta^r_{F\otimes G}\cong \Delta^r_G\Delta^r_F
\]
It remains to give a categorical analog of the relation $\Delta_x m= m {(\Delta)^2}_{\Delta(x)}$. 
Let $\bar{c}$ be the arrow $(\fset{2},\fset{2})\rightarrow(\fset{4})$ in $\FinDisj$ given by the maps $\mapstack{1\mapsto 1}{2\mapsto 3},\mapstack{1\mapsto 2}{2\mapsto 4}$. For $\Phi\in\CS{\fset{2}}$ we define the functor $i_\Phi:\CS{\fset{2}}\rightarrow \CS{\fset{4}}$ by
\[i_{\Phi}(X)=\Hop(\bar{c})(\Phi,X)\]
Define adjoint functors $m^2_\Phi\dashv \Delta^2_\Phi$ as follows:
\begin{align*}
m^2_\Phi&=m^2\circ i_\Phi\\
(\Delta^r)^2_\Phi &=j^r_\Phi\circ(\Delta^r)^2
\end{align*}
then in this notation we have 

\begin{Theorem}
\label{th:deltam}
There is a canonical isomorphism 
\begin{equation}
\label{deltam}
\Delta^r_F m\cong m\circ (\Delta^r)^2_{\Delta^l(F)}
\end{equation}
coming from the SSH structure on $\CCC$.
\end{Theorem}

\begin{Corollary}
We have a canonical isomorphism \[
\Delta^r_Fm_G=\Delta^r_F m\circ i_G \cong m\circ (\Delta^r)^2_{\Delta(F)}\circ i_G
\]
\end{Corollary}

\begin{proof}
Note first that the isomorphism \eqref{deltam} can be represented as a square

\begin{equation}
\label{heisbcsquare}
\tik
\CS{2} \arrow[]{r}[above]{m}  \arrow[shorten >=0.5cm,shorten <=0.5cm,Rightarrow]{dr}[above,sloped]{\sim} & \CCC   \\
\CS{2} \arrow[]{u}{(\Delta^r)^2_{\Delta^l(F)}} \arrow[]{r}[below]{m} & \CCC \arrow[]{u}[right]{\Delta^r_F}
\tak
\end{equation}

The form of this square suggests that we should try to construct it as the right mate of a square of the form
\begin{equation}
\label{heisquare}
\tik
\CS{2} \arrow[]{d}[left]{m^2_{\Delta^l(F)}} \arrow[]{r}[above]{m} & \CCC \arrow[shorten >=0.4cm,shorten <=0.4cm,Rightarrow]{dl}[above,sloped]{\alpha}  \arrow[]{d}[right]{m_F}\\
\CS{2} \arrow[]{r}[below]{m}& \CCC 
\tak
\end{equation}
satisfying the Beck-Chevalley condition. Our objective is to construct such square from the SSH structure on $\CCC$. We start by rewriting it as a composition of squares:
\[
\tik
\CS{2} \arrow[]{d}[left]{i_{\Delta^l(F)}} \arrow[]{r}{m} & \CCC \arrow[shorten >=0.4cm,shorten <=0.4cm,Rightarrow]{dl}[above,sloped]{\beta}  \arrow[]{d}{i_F}\\
\CS{4} \arrow[]{r}{\overline m} \arrow[]{d}[left]{m^2}& \CS{2} \arrow[]{d}{m}\arrow[shorten >=0.4cm,shorten <=0.4cm,Rightarrow]{dl}[above,sloped]{\sim}\\
\CS{2} \arrow[]{r}{m} & \CCC
\tak
\]
where $\overline m$ corresponds to the maps of sets $\mapstack{1\mapsto 1}{3\mapsto 1},\mapstack{2\mapsto 2}{4\mapsto 2}$ and the lower square is the image of a Cartesian square under the SSH functor. The right mate of the lower square is therefore invertible, and hence it is enough to construct a square $\beta$ satisfying the Beck-Chevalley condition. 

The square $\beta$ can be obtained as a composition in $\CatAdTen$:
\begin{equation}
\label{bcdecomposition}
\stik{1}{
(\CCC^{\otimes 2})\ar{d}[left]{(\Delta^l(F),\Id)} \ar{r}{m} \& (\CCC) \ar{d}{(F,\Id)} \arrow[shorten >=0.4cm,shorten <=0.4cm,Rightarrow]{dl}[above,sloped]{(\eta,\Id)}\\
(\CCC^{\otimes 2},\CCC^{\otimes 2}) \ar{d}[left]{\Hop(\bar{c})} \ar{r}{(m,m)} \& (\CCC,\CCC) \ar{d}{\Hop(c)} \arrow[shorten >=0.4cm,shorten <=0.4cm,Rightarrow]{dl}[above,sloped]{\sim}\\
(\CCC^{\otimes 4}) \ar{r}{\overline{m}}  \& (\CCC^{\otimes 2})
}
\end{equation}

We will now prove that each of these squares satisfies the Beck-Chevalley condition, and hence also their composition does. In \autoref{heuristics} we explain how the decomposition of $\beta$ above can be obtained using the equivalence of Beck-Chevalley squares and \comma squares in $\CatAd$ from \autoref{sec:2Cartesian}.

The top square is given by two squares (as described in  \autoref{def:CatAdTen}):
\begin{align*}
\stik{1}{
\Vect\ar{d}[left]{\Delta^l(F)} \ar[equal]{r} \& \Vect \ar{d}{F} \arrow[shorten >=0.4cm,shorten <=0.4cm,Rightarrow]{dl}[above,sloped]{\eta}\\
\CCC^{\otimes 2} \ar{r}{m} \& \CCC}
&\stik{1}{}
&
\stik{1}{\CCC^{\otimes 2}\ar[equal]{d} \ar{r}{m} \& \CCC \ar[equal]{d} \arrow[shorten >=0.4cm,shorten <=0.4cm,equal]{dl}\\
\CCC^{\otimes 2} \ar{r}{m} \& \CCC }
\end{align*}

where $\eta:F\rightarrow m\Delta^l(F)$ comes from the unit of the adjunction.

Both of these squares satisfy the left Beck-Chevalley condition. For the left square this follows from the definition of adjunction of functors, and the right square is degenerate.

The bottom square in \eqref{bcdecomposition} is the image of the Cartesian square in $\FinDisj$
\[
\stik{1}{
(\fset{2},\fset{2}) \ar{d}[left]{(\bar{c})} \ar{r}{(m,m)} \& (\fset{1},\fset{1}) \ar{d}{(c)}\\
(\fset{4}) \ar{r}{\overline{m}}  \& (\fset{2})
}
\]

under $\Hop$, where $c$ is the pair of maps $1\mapsto 1,1\mapsto 2$ and $\bar{c}$ is computed accordingly as the pullback.

We therefore have left to prove:

\begin{Lemma}
\label{lem:delignetensorbc}
The square
\[
\stik{1}{
(\fset{2},\fset{2}) \ar{d}[left]{(\bar{c})} \ar{r}{(m,m)} \& (\fset{1},\fset{1}) \ar{d}{(c)}\\
(\fset{4}) \ar{r}{\overline{m}}  \& (\fset{2})
}
\]
in $\FinDisj$ maps to a square in $\CatAdTen$ that satisfies the left Beck-Chevalley condition.
\end{Lemma}

\begin{proof}
Note that the verticals are maps presenting a set as a disjoint union of two sets, and hence are coCartesian arrows in $\FinDisj$. By assumption they must go to maps presenting the Deligne tensor, so it is enough for us to prove the following more general assertion:

Let 
\[
\stik{1}{
(A,B) \ar{d} \ar{r} \& (C,D) \ar{d} \ar[Rightarrow,shorten <=0.7em,shorten >=0.7em]{dl}[above,sloped]{\alpha} \\
(A\otimes B) \ar{r}  \& (C\otimes D)
}
\]
be a square in $\CatAdTen$ where the verticals are presentations of the respective Deligne tensors (and so we denote the target category as the tensor), and $\alpha$ is invertible, then it satisfies the left Beck-Chevalley condition.

The idea of the proof is that the square with $\alpha$ should be a coCartesian arrow in the category of arrows fibered over the category of arrows of $\FinSet_*$, and so its mate should be a coCartesian arrow in the opposite category. This is because taking the mate is in some sense an invertible operation. Then the invertibility of $\alpha$ implies the invertibility of the mate.

More concretely, consider an incomplete cube
\[
\stik{1}{
{} \& (A,B) \ar{rr}[name=topback,below]{} \ar{dd}[name=backleft,sloped,xshift=-1.5pc]{} \& {} \& (C,D) \ar{dd}\\
(A,B)\ar[crossing over]{rr}[name=topfront,above]{}\ar{dd}[name=frontleft]{} \ar[equals]{ur}[name=topleft,below]{} \& {} \& (C,D) \ar{dd}[name=frontright,sloped,xshift=1pc]{} \ar[equals]{ur}[name=topright,below]{} \& {}\\
{} \& (E) \ar{rr} \& {} \& (F)\\
(A\otimes B) \ar{rr}[name=bottomfront,below]{} \& {} \& (C\otimes D) 
\arrow[Rightarrow,to path={(frontright) to[bend right] (bottomfront)}, "\alpha"]{}{}
\arrow[equal,to path={(topright) to[bend left] (topback)}]{}{}
\arrow[Rightarrow,to path={(topback) to[bend left] (backleft)}, "\beta"]{}{}
\latearrow{/tikz/commutative diagrams/crossing over}{2-3}{4-3}{}
\latearrow{/tikz/commutative diagrams/crossing over,/tikz/commutative diagrams/equal}{2-3}{1-4}{}
}
\]
where the top is degenerate. We note that the property of the Deligne tensor allows us, first of all, to fill in the sides with an isomorphism, and given that, to fill in the bottom in a unique way. So we get a cube
\[
\stik{1}{
{} \& (A,B) \ar{rr}[name=topback,below]{} \ar{dd}[name=backleft,sloped,xshift=-1.5pc]{} \& {} \& (C,D) \ar{dd}\\
(A,B)\ar[crossing over]{rr}[name=topfront,above]{}\ar{dd}[name=frontleft]{} \ar[equals]{ur}[name=topleft,below]{} \& {} \& (C,D) \ar{dd}[name=frontright,sloped,xshift=1pc]{} \ar[equals]{ur}[name=topright,below]{} \& {}\\
{} \& (E) \ar{rr} \& {} \& (F)\\
(A\otimes B) \ar{ur}[name=bottomleft,sloped,yshift=-0.2pc]{} \ar{rr}[name=bottomfront,below]{} \& {} \& (C\otimes D) \ar{ur}
\arrow[equal,to path={(topright) to[bend left] (topback)}]{}{}
\arrow[Rightarrow,to path={(frontright) to[bend right] (bottomfront)}, "\alpha"]{}{}
\arrow[Rightarrow,crossing over,to path={(frontright) to[bend right] (topright)}]{}{}
\arrow[Rightarrow,to path={(bottomfront) to[bend right] (bottomleft)},"f"]{}{}
\arrow[Rightarrow,to path={(bottomleft) to[bend left] (backleft)}]{}{}
\arrow[Rightarrow,to path={(topback) to[bend left] (backleft)}, "\beta"]{}{}
\latearrow{/tikz/commutative diagrams/crossing over}{2-3}{4-3}{}
\latearrow{/tikz/commutative diagrams/crossing over}{2-1}{2-3}{}
\latearrow{/tikz/commutative diagrams/crossing over,/tikz/commutative diagrams/equal}{2-3}{1-4}{}
}
\]
Passing to the left mate, as described in \autoref{sec:BC}, we get a cube
\[
\stik{1}{
{} \& (A,B) \ar[equals]{rr}[name=topback,below]{} \ar{dd}[name=backleft,sloped,xshift=-1.5pc]{} \& {} \& (A,B) \ar{dd}\\
(C,D)\ar[equals,crossing over]{rr}[name=topfront,above]{}\ar{dd}[name=frontleft]{} \ar[dashed,red,thick]{ur}[name=topleft,below]{} \& {} \& (C,D) \ar{dd}[name=frontright,sloped,xshift=1pc]{} \ar[dashed,red,thick]{ur}[name=topright,below]{} \& {}\\
{} \& (A\otimes B) \ar{rr} \& {} \& (E)\\
(C\otimes D) \ar[dashed,red,thick]{ur}[name=bottomleft,sloped,yshift=-0.2pc]{} \ar{rr}[name=bottomfront,below]{} \& {} \& (F) \ar[dashed,red,thick]{ur}
\arrow[equal,to path={(topright) to[bend left] (topback)}]{}{}
\arrow[Rightarrow,to path={(frontright) to[bend right] (bottomfront)}]{}{}
\arrow[Rightarrow,red,crossing over,to path={(frontright) to[bend right] (topright)}, "\beta_L"]{}{}
\arrow[Rightarrow,red,to path={(bottomfront) to[bend right] (bottomleft)},"f_L"]{}{}
\arrow[Rightarrow,red,to path={(bottomleft) to[bend left] (backleft)}, "\alpha_L"]{}{}
\arrow[Rightarrow,to path={(topback) to[bend left] (backleft)}]{}{}
\latearrow{/tikz/commutative diagrams/crossing over}{2-3}{4-3}{}
\latearrow{/tikz/commutative diagrams/crossing over,/tikz/commutative diagrams/equal}{2-1}{2-3}{}
\latearrow{/tikz/commutative diagrams/crossing over,dashed,red,thick}{2-3}{1-4}{}
}
\]
Now, if the unit/counit data is fixed, the operations of left mate and right mate are inverse to each other, so we get that $\alpha_L$ has a property that (up to choosing isomorphisms coming from the Deligne tensor property for the front and back - noting also that these faces are not touched by the mate operations) any square
\[
\stik{1}{
(C,D) \ar{d} \ar{r} \& (A,B) \ar{d} \ar[Rightarrow,shorten <=0.7em,shorten >=0.7em]{dl}[above,sloped]{\gamma} \\
(F) \ar{r}  \& (E)
}
\]
factors uniquely through $\alpha_L$. It is now straitforward to see that this implies $\alpha_L$ is an isomorphism.
\end{proof}

We have proved that both squares and therefore their composition satisfies the left Beck-Chevalley condition. It follows (see \autoref{app:BeckChevalley}) the composition square also satisfies the right Beck-Chevalley condition since for this square both left and right mates are defined. The right Beck-Chevalley condition is what we need to obtain the isomorphism \eqref{deltam}.
\end{proof}


\subsection{The Heisenberg double category and a conjectural generalization of the Heisenberg relation}
\label{heuristics}
\begin{Note}
Throughout this section we will write $\Delta$ for $\Delta^l$ - the left adjoint of the multiplication.
\end{Note}

In this section we outline a more conceptual  approach for constructing and working with the Beck-Chevalley squares of the form

\begin{equation}
\stik{1}{
\label{originalheissquare}
\CS{2} \arrow[]{d}[left]{i_{\Delta(F)}} \arrow[]{r}[above]{m} \& \CCC \arrow[shorten >=0.4cm,shorten <=0.4cm,Rightarrow]{dl}[above,sloped]{\beta}  \arrow[]{d}[right]{i_F}\\
\CS{4} \arrow[]{r}[below]{\overline{m}} \& \CS{2} 
}
\end{equation}
in an SSH category $\CCC$. The key idea is to use the equivalence between Beck-Chevalley and \comma squares from \autoref{sec:2Cartesian} and think about the square $\beta$ as a pullback of its bottom right corner. This section describes the method for computing this pullback which in particular provides us with the decomposition \eqref{bcdecomposition} into Beck-Chevalley squares used in the previous section to construct a canonical 
lift of the Heisenberg double relation. Our hope is that the construction described in this section can be generalized and made precise in various contexts to produce new canonical categorifications and therefore is of independent interest. 

We start from noting that  \autoref{lem:delignetensorbc} makes it reasonable to expect that images under $\Hop$ of a certain class of Cartesian squares in $\FinDisj$ (not necessarily lying in the fiber over $\pset{1}$) are Beck-Chevalley squares. We would like to state this in the language of preserving \comma squares. We introduce a notion of \emph{active} square in $\FinDisj$: \begin{Definition}
A map $f:\pset{n}\rightarrow \pset{m}$ in $\FinSet_*$ is called \emph{active} if $f^{-1}\{*\}=\{*\}$.
\end{Definition}

Consider a square in $\FinDisj$:
\[
\stik{1}{
(A_i)_{i\in I} \ar{r}{\varphi} \ar{d}{\psi} \& (B_j)_{j\in J} \ar{d}{\tau} \\
(C_k)_{k\in K} \ar{r}{\nu} \& (D_l)_{l\in L}
}
\]
and suppose that it lies over a square of active maps in $\FinSet_*$. We say in this case that it is an \emph{active} square. The subcategory of active maps in $\FinSet_*$ is closed on pullbacks in $\FinSet_*$, so a Cartesian square of active maps is just a usual Cartesian square in the subcategory $(\FinSet_*)_{ac}\cong\FinSet$. The following gives an easy way of computing pullbacks of active maps.
\begin{Lemma} \label{lem:commafinset} 
An active square is Cartesian in $\FinDisj$, iff the square it lies over, \stik{0.7}{\pset{I}\ar{r}\ar{d} \& \pset{J}\ar{d}\\
\pset{K}\ar{r} \& \pset{L}}, is Cartesian in $\FinSet_*$, and all the squares of sets involved in it are Cartesian in $\FinSet$.
\end{Lemma}

\begin{proof}
We may assume that $L$ has one point, and then we have an explicit construction for a Cartesian square with given bottom right corner. Namely, we consider sets $T_{k,j},k\in K, j\in J$ given by $T_{k,j}=B_j\times_D C_k$ with the obvious projections to $B_\bullet,C_\bullet$. The lemma then follows immediately.
\end{proof}

\begin{Conjecture} \label{prop:generalcomma} \label{conj:activecomma}
An SSH functor $\Hop$ takes active \comma squares in $\FinDisj$ to (necessarily active) \comma squares in $\CatAdTen$.
\end{Conjecture}

We believe that the an ideologically correct proof of this general statement should follow from factoring $\Hop$ through a functor to a category of algebra objects in $\CatAd$. This kind of approach necessitates a construction and careful consideration of such a framework which is outside the scope of this article. 

Instead, we restrict ourselves to proving this for the specific square needed for the construction of Heisenberg relations in \autoref{lem:delignetensorbc}. However, writing the statement in this general form allows us to generalize our approach as follows:

Restrict $\Hop$ over the subcategory of active maps in $\FinSet_*$ to get a functor 
\[
\Hop_{ac}:(\FinDisj)_{ac}\rightarrow(\CatAdTen)_{ac}
\]
and consider a bicategory $\Heis(\Hop)$  defined by the following comma square of bicategories

\begin{equation}\label{heisenbergsquare}
\stik{1}{\Heis(\Hop) \ar{r} \ar{d}[left]{p} \& \CatAd \ar{d}{i} \ar[Rightarrow,shorten <=0.7em,shorten >=0.7em]{dl}[above,sloped]{} \\
(\FinDisj)_{ac} \ar{r}{\Hop_{ac}} \& (\CatAdTen)_{ac}
}
\end{equation}
where $\CatAd\xrightarrow{i}(\CatAdTen)_{ac}$ is the obvious imbedding. We call the bicategory $\Heis(\Hop)$ the \emph{categorical Heisenberg double} of an SSH functor $\Hop$.

The objects and 1-morphisms in $\Heis(\Hop)$ can be described as follows
\begin{itemize}
\item Objects are pairs $(C,S=(S_k)_{k\in K})$ with $C\in\CatAd$ and $S\in\FinDisj$, along with a morphism $i(C)\xrightarrow{a}\Hop(S)$.
\item A morphism $(C,S)\rightarrow(D,T)$ is a pair $(C\xrightarrow{\varphi}D,S\xrightarrow{f}T)$, and a square in $\CatAdTen$\[
\stik{1}{
i(C) \ar{d} \ar{r}{i(\varphi)} \& i(D) \ar{d} \ar[Rightarrow,shorten <=0.7em,shorten >=0.7em]{dl}[above,sloped]{\alpha}\\
\Hop(S) \ar{r}[below]{\Hop(f)} \& \Hop(T)
}
\]
We call $\varphi$ the \emph{functor part} of the morphism, and $\alpha$ the \emph{transformation part} of the morphism.
\item 2-morphisms are defined in the obvious way.
\end{itemize}

In particular, let $(\CCC,(S_i),a)$ be an object in $\Heis(\Hop)$ where the map $a$ sits over an injection $\pset{1}\rightarrow\pset{n}$ which sends $1$ to $k$. Let $\Hop(S)=(\Hop(S_1),\ldots,\Hop(S_n))$, then $a$ amounts to a choice of elements in all but the $k^{\text{th}}$ category, and a map $\CCC\rightarrow \Hop(S_k)$.

Both functors forming the bottom right corner of the pullback square \eqref{heisenbergsquare} preserve \comma squares (whose definition closely resembles that of a weighted limit) hence it is reasonable to expect

\begin{Conjecture}
\label{conj:heiscomma}
The functors out of $\Heis(\Hop)$ preserve \comma squares.
\end{Conjecture}

Let us use the above to provide a canonical construction of the square $\beta$ from the beginning of this section. We will compute the pullback of a preimage of its lower right corner in $\Heis(\Hop)$. If we believe the conjectures in this section, the resulting $\comma$ square in $\Heis(\Hop)$ - which is a canonical object - should then project to an \comma square in $\CatAd$, which is the same as a Beck-Chevalley square.

We have the obvious preimages for the maps $i_F,\overline{m}$ (abbreviating $(\fset{k})$ for $\Hop(\fset{k})$):
\begin{align*}
\stik{1}{
(\fset{1}) \ar{d}[left]{(F,\Id)} \ar{r}{i_F} \& (\fset{2}) \ar[equal]{d} \ar[Rightarrow,shorten <=0.7em,shorten >=0.7em]{dl}\\
(\fset{1},\fset{1}) \ar{r}[below]{\Hop(c)} \& (\fset{2})
}
& \stik{1}{}
& \stik{1}{
(\fset{4}) \ar[equal]{d} \ar{r}{\overline{m}} \& (\fset{2}) \ar[equal]{d} \ar[equal,shorten <=0.7em,shorten >=0.7em]{dl}\\
(\fset{4}) \ar{r}{\overline{m}} \& (\fset{2})
}
\end{align*}
where $c=(1\mapsto 1,1\mapsto 2)$ and in the left square the 2-morphism is the unitor morphism $\Id\circ(f\circ g)\rightarrow  f\circ g$.

So we consider the following bottom right corner in $\Heis(\Hop)$:
\[
\stik{1}{
{} \& {} \& {} \& (\fset{1},\fset{1}) \ar{dd}\\
{}  \& {} \& (\fset{1}) \ar{dd}[right,yshift=1pc]{i_F} \ar{ur}{(F,\Id)} \& {}\\
{} \& (\fset{4}) \ar{rr}[xshift=-2pc]{\overline{m}} \& {} \& (\fset{2})\\
(\fset{4}) \ar{rr}{\overline{m}} \ar[equal]{ur} \& {} \& (\fset{2}) \ar[equal]{ur}
\latearrow{commutative diagrams/crossing over}{2-3}{4-3}{}
}
\]
And we want to complete it to an \comma square. By \autoref{conj:heiscomma} the back face should come from an \comma square in $\FinDisj$, so first we compute the product in the back (as in \autoref{lem:commafinset}) to get
\[
\stik{1}{
{} \& (\fset{2},\fset{2}) \ar{dd}[left]{(\bar{c})} \ar{rr}{(m,m)}\& {} \& (\fset{1},\fset{1}) \ar{dd}{(c)}\\
{}  \& {} \& (\fset{1}) \ar{dd}[right,yshift=1pc]{i_F} \ar{ur}{(F,\Id)} \& {}\\
{} \& (\fset{4}) \ar{rr}[xshift=-2pc]{\overline{m}} \& {} \& (\fset{2})\\
(\fset{4}) \ar{rr}{\overline{m}} \ar[equal]{ur} \& {} \& (\fset{2}) \ar[equal]{ur}
\latearrow{commutative diagrams/crossing over}{2-3}{4-3}{}
}
\]
where $\bar{c}=(\mapstack{1\mapsto 1}{1\mapsto 3},\mapstack{1\mapsto 2}{1\mapsto 4})$. Note that the back face is the image of an active \comma square under $\Hop$, so is an \comma square in $\CatAd$ by \autoref{conj:activecomma}.

To proceed we need to make an ansatz regarding \comma squares:

Suppose that we have completed what we have to a cube representing an \comma square in $\Heis(\Hop)$. 

\[
\stik{1}{
{} \& (\fset{2},\fset{2}) \ar{dd}[left,yshift=-1pc]{(\bar{c})} \ar{rr}{(m,m)}\& {} \& (\fset{1},\fset{1}) \ar{dd}{(c)}\\
(?) \ar{dd}[left]{?} \ar[crossing over]{rr}[xshift=1pc]{?} \ar{ur}{?}  \& {} \& (\fset{1}) \ar{dd}[right,yshift=1pc]{i_F} \ar{ur}{(F,\Id)} \& {}\\
{} \& (\fset{4}) \ar{rr}[xshift=-2pc]{\overline{m}} \& {} \& (\fset{2})\\
(\fset{4}) \ar{rr}{\overline{m}} \ar[equal]{ur} \& {} \& (\fset{2}) \ar[equal]{ur}
\latearrow{commutative diagrams/crossing over}{2-3}{4-3}{}
}
\]

The right side of the cube represents a morphism in $\Heis(\Hop)$ which has a fully faithful functor part and an invertible transformation part. This is reminiscient of a monomorphism, and so it seems reasonable to conjecture that the left face, being a "pullback" of the right along the bottom, should also be of this type. Given this, and that by \autoref{conj:heiscomma} the front face is an \comma square in $\CatAd$, an extension of the pasting lemma for pullbacks to the \comma case would give us that the top square must also be \comma.

Therefore, to proceed we compute the product on the top. Consider the corner 
\[
\stik{1}{
{} \& (\fset{1}) \ar{d}{(F,\Id)} \\
(\fset{2},\fset{2}) \ar{r}{(m,m)} \& (\fset{1},\fset{1})
}
\]
Since the element in the bottom right is over $\pset{2}$, a square with this corner is given by two squares, with corners:
\begin{align*}
\stik{1}{
{} \& (\fset{1}) \ar[equal]{d} \\
(\fset{2}) \ar{r}{(m)} \& (\fset{1})
}
&\stik{1}{}
&
\stik{1}{
{} \& () \ar{d}{(F)} \\
(\fset{2}) \ar{r}{(m)} \& (\fset{1})
}
\end{align*}
and we should compute the product separately for each one.
The first corner fits in a degenerate square, and any degenerate square is \comma. 
Let us compute the pullback for the second corner:

Suppose we have a square with this bottom right corner, i.e.
\[
\stik{1}{
(\CCC_i) \ar{r} \ar{d}{\varphi} \& () \ar{d} \\
(\fset{2}) \ar{r}{m} \& (\fset{1})
}
\]

First, the map $\varphi$ must sit over a map which sends everything to the base point, in order for the whole square to sit over a commutative square. So we may assume that $(C_i)=()$ and that $\varphi$ is the selection of an element of $\fset{2}$. In all, we can assume that the square is of the form 
\[
\stik{1}{
() \ar{d}[left]{X} \ar[equal]{r} \& () \ar{d}{F} \ar[Rightarrow,shorten <=0.7em,shorten >=0.7em]{dl}[above,sloped]{\gamma}\\
(\fset{2}) \ar{r}{m} \& (\fset{1})
}
\]

i.e. a map $F\xrightarrow{\gamma}mX$. If this square is to be \comma, the map $\gamma$ must satisfy the property that for any map $F\xrightarrow{\psi}mY$ we have a unique map $X\xrightarrow{\tau}Y$ such that $\psi=m(\tau)\circ\gamma$, i.e. that $X$ together with $\gamma$ is a local left adjoint for $m$ at $F$. Since $m$ has a global left adjoint, this implies that $X$ and $\gamma$ are canonically isomorphic to $\Delta(F)$ and $F\xrightarrow{\eta}m\Delta(F)$.

So now we have
\[
\stik{1}{
{} \& (\fset{2},\fset{2}) \ar{dd}[left,yshift=-1pc]{(\bar{c})} \ar{rr}{(m,m)}\& {} \& (\fset{1},\fset{1}) \ar{dd}{(c)}\\
(\fset{2}) \ar[crossing over]{rr}[xshift=1pc]{m} \ar{ur}{(\Delta(F),\Id)}  \& {} \& (\fset{1}) \ar{dd}[right,yshift=1pc]{i_F} \ar{ur}{(F,\Id)} \& {}\\
{} \& (\fset{4}) \ar{rr}[xshift=-2pc]{\overline{m}} \& {} \& (\fset{2})\\
(\fset{4}) \ar{rr}{\overline{m}} \ar[equal]{ur} \& {} \& (\fset{2}) \ar[equal]{ur}
\latearrow{commutative diagrams/crossing over}{2-3}{4-3}{}
}
\]
and we complete to a cube by setting the missing 1-morphism to be the composition $(\bar{c})\circ(\Delta(F),\Id)=i_{\Delta(F)}$; the left face to be the unitor isomorphism to the composition (to be compatible with our ansatz); and the front face to be the unique 2-morphism that makes the cube commute as in \autoref{app:cubecommutes}.

In all we have the cube which should represent an \comma square in $\Heis(\Hop)$
\[
\stik{1}{
{} \& (\fset{2},\fset{2}) \ar{dd}[left,yshift=-1pc]{(\bar{c})} \ar{rr}{(m,m)}\& {} \& (\fset{1},\fset{1}) \ar{dd}{(c)}\\
(\fset{2}) \ar{dd}[left]{i_{\Delta(F)}} \ar[crossing over]{rr}[xshift=1pc]{m} \ar{ur}{(\Delta(F),\Id)}  \& {} \& (\fset{1}) \ar{dd}[right,yshift=1pc]{i_F} \ar{ur}{(F,\Id)} \& {}\\
{} \& (\fset{4}) \ar{rr}[xshift=-2pc]{\overline{m}} \& {} \& (\fset{2})\\
(\fset{4}) \ar{rr}{\overline{m}} \ar[equal]{ur} \& {} \& (\fset{2}) \ar[equal]{ur}
\latearrow{commutative diagrams/crossing over}{2-3}{4-3}{}
}
\]
Now, the front face is exactly the square $\beta$ and assuming \autoref{conj:heiscomma}, it is an \comma square in $\CatAd$. 

Independently of the conjectures, from the commutativity of the above cube we get the decomposition \eqref{bcdecomposition}, used to lift the Heisenberg relation in the previous section.In this sense one can consider the material in this chapter as providing a heuristic method for constructing such canonical squares.


\subsection{A categorification of the Fock space action of the infinite dimensional Heisenberg algebra}
\label{ssec:FockP}
In this section we consider the example of the SSH category $\PP$ of polynomial functors and apply our general construction of a categorical Heisenberg double to this example. This gives us a canonically constructed categorification of the Fock space representation of infinite dimensional Heisenberg algebra, which is the Heisenberg double associated to the Grothendieck group $K(\PP)$.

Consider the coCartesian arrow $c=(c_1,c_2):(\{1\},\{2\}) \cartesianarrow \{1,2\}$ in $\FinDisj$. Then, from the definition of the functor $\Hop$ giving the SSH structure on $\PP$, for any $F\in\PP$ the functor $i_F: \PP \rightarrow\PPS{\fset{2}}$ is given by the formula:
\begin{align*}
i_FX(V)=F(c_0^*V)\otimes X(c_1^*V)
\end{align*}
where $X \in \PP, V\in \Sh(\{1,2\})$.Its left and right adjoints are both given by the functor $j_F$:
\begin{align*}
j_F:\PPS{[2]} & \rightarrow\PP, &  j_F\Phi(V)=\Hom_\PP(F,\Phi(\boxtimes_c(-,V))
\end{align*}
where $\boxtimes_c(W, V)$ is the sheaf defined by $c_0^*\boxtimes_c(W, V)=W,c_1^*\boxtimes_c(W, V)=V$. 
The functors $m_F,\Delta_F$ are defined by composing
\begin{align*}
m_{F}&=m\circ i_F & \Delta_F&=j_F\circ\Delta
\end{align*} 

Note that these functors are biadjoint, so we will not distinguish between $\Delta^r_F$ and $\Delta^l_F$ and just write $\Delta_F$. 

The isomorphisms  
\begin{align*}
m_Fm_G\cong m_{F\otimes G} \\
\Delta_G\Delta_F\cong \Delta_{F\otimes G}
\end{align*}
come from the associator of the monoidal structure on the category $\Vect$. As we saw in section \autoref{ssec:CatHeisenbergConstructions}, the non-trivial part of the isomorphism $\Delta_F m\cong m \Delta^2_{\Delta(F)}$ is given by the mate of the top square in \eqref{bcdecomposition}. We want to check that the mate is invertible in the case of $\PP$. Since we are now working with a concrete example, we will just compute the mate using the formula from \autoref{sec:BC} and check it is an invertible morphism.

Consider the arrow $\bar{c}: (\fset{2},\fset{2})\rightarrow(\fset{4})$ in $\FinDisj$ given by the maps $\bar{c}_0:\mapstack{1\mapsto 1}{2\mapsto 3},\bar{c}_1:\mapstack{1\mapsto 2}{2\mapsto 4}$ and let
\begin{align*}
i_{\Phi}:\PPS{\fset{2}} & \rightarrow\PPS{\fset{4}} &i_{\Delta \Phi}\Psi(V)&=\Phi(\bar{c_0}^*V)\otimes \Psi(\bar{c_1}^*V)\\
j_{\Phi}:\PPS{\fset{4}} & \rightarrow\PPS{\fset{2}}, &  j_\Phi\Omega(V)&=\Hom(\Phi,\Omega(\boxtimes_{\bar{c}}(-,V))))
\end{align*}
The functors $i_\Phi, j_\Phi$ are biadjoint. Using the adjunction morphisms $\eta:\Id\rightarrow j_Fi_F$ and $\epsilon:i_{\Delta(F)}j_{\Delta(F)}\rightarrow\Id$, the mate of the top square in \eqref{bcdecomposition} is given by:
\begin{equation}
\label{rightmateofdoubleimage}
mj_{\Delta(F)}\xrightarrow{\eta}j_Fi_Fmj_{\Delta(F)}\xrightarrow{\alpha}j_F\overline{m} i_{\Delta(F)}j_{\Delta(F)}\xrightarrow{\epsilon}j_F\overline{m}
\end{equation}

The maps $\eta$ and $\epsilon$ are given by the following formulas:
\begin{equation}\label{i1j1runit}
\begin{gathered}
\begin{aligned}
G(V) \xrightarrow{\eta} j_Fi_F(G)(V) &=\Hom(F,F(-)\otimes G(V)) \\ &=\End(F) \otimes G(V)\end{aligned} \\  
x \mapsto \Id \otimes x\
\end{gathered}
\end{equation}
\begin{equation}\label{i1j1rcounit}
\begin{gathered}
i_{\Delta F}j_{\Delta F}(\Omega)(V)=\Delta F(\bar{c}_0^*V)\otimes \Hom (\Delta F,\Omega(\boxtimes_{\bar{c}}(-,\bar{c}_1^*V))) \xrightarrow{\epsilon} \Omega(V) \\ 
x\otimes\tau \mapsto \tau_{\bar{c}_0^*V}(x)
\end{gathered}
\end{equation}
(note that for any sheaf $V\in\Sh(\fset{4})$ the sheaf $\boxtimes_{\bar{c}}(\bar{c}_0^*V, \bar{c}_1^*V)$ is canonically isomorphic to $V$)

Take $\Omega\in\PPS{4},U\in\Sh(\fset{1})$ and let $a:\fset{2}\rightarrow \fset{1}$. Then 
\begin{equation}
mj_{\Delta(F)}\Omega(U)=\Hom(\Delta F,\Omega(\boxtimes_{\bar{c}}(-,\bar{c}_1^*U)))
\end{equation}

We want to take $x\in mj_{\Delta(F)}\Omega(U)$ and follow it through the map. Denote by $W$ some vector space and $v$ an element in $F(W)$, then
\begin{gather*}
\eta(x)_W=(v\mapsto v\otimes x)\\
\alpha(\eta(x))_W=(v\mapsto \eta_L(v)\otimes m(x)) \\
\epsilon(\alpha(\eta(x)))_W=(v\mapsto m(x)(\eta_L(v)))
\end{gather*}

where $\eta_L$ is the map $F\rightarrow m\Delta F$ from the adjunction $(m,\Delta)$ (see \autoref{app:pullpushadjsets} for details) and $m(x)$ is the application of $m$ to $x$, which is a map $m\Delta F$ to $m\Omega(\boxtimes_{\bar{c}}(-,\bar{c}_1^*U))$. Altogether, we have computed that the mate takes the map $x$ to $m(x)$ precomposed with the unit $\eta_L$, but this is exactly the construction of the isomorphism \[
\Hom(\Delta F,\Omega(\boxtimes_{\bar{c}}(-,\bar{c}_1^*U)))\xrightarrow{\sim}\Hom(A,m\Omega(\boxtimes_{\bar{c}}(-,\bar{c}_1^*U)))
\]
for the adjoint pair $m,\Delta$, so it is invertible.


\section{Equivariant polynomial functors}
\label{sec:pequivariant}
In this section we extend our example of polynomial functors to a class of examples which we call $G$-equivariant polynomial functors. This is the direct categorical analog of the example of representations of wreath products $S_n[G]$ considered by Zelevinsky in \cite{Zelbook}.


\subsection{Definition}
\label{ssec:defpequivariant}
Recall that we defined the SSH structure on $\PP$ in \autoref{sec:selfadjointhopf} by giving a symmetric monoidal functor $\FinDisj\xrightarrow{\Hop}\CatAdTen$. On objects it was defined by 
\[
\Hop(S)=\PPS{S}
\]
where $\PPS{S}$ was defined to be the category of polynomial functors from $\Sh(S)$ to $\Vect$.

Fix a finite group $G$. We define a functor $\Hop_G$ in the same way we defined $\Hop$, but replacing $\Sh(S)$ with $\Sh_G(S)$, i.e. sheaves of $G$-representations on $S$.
\begin{Remark}
More naturally, the assignment $S\mapsto\Sh_G(S)$ can be split into two pieces:
\begin{enumerate}
\item The imbedding $\FinSet\hookrightarrow\GSet$ given by $S\mapsto S$ with the trivial $G$ action.
\item The functor $\GSet\xrightarrow{}\Vect\Cat$ which sends a $G$-set $X$ to the category of sheaves on the groupoid $X/G$.
\end{enumerate}
It might be interesting to consider on its own the functor $\GSet\rightarrow\Vect\Cat$.
\end{Remark}

Define $\PP_G$ to be $\Hop_G(\fset{1})$.


\subsection{A manifestation of Zelevinsky's decomposition Theorem}
Zelevinsky's main theorem about PSH algebras, is that they are all tensor products of many copies of the algebra $\Lambda$ of symmetric functions. His proof is somewhat combinatorial, and gives the morphism to the tensor product only up to a non canonical choice.

In this section we would like to show a categorical analog of it, where we obtain a canonical equivalence (of SSH categories) to a tensor power of $\PP$. We conjecture that a similar result holds in general for SSH categories.

Consider the functor $Y:\Sh_G(S)\rightarrow\Sh(S\times \Irr G)$ given by \[
Y(V)(s,\rho)=\Hom(\rho,Y_s)
\]
Since we are working in characteristic zero, this functor is an equivalence. Hence it induces an equivalence \begin{align*}
\PPS{(S\times \Irr G)}&=\PolyFun(\Sh(S\times \Irr G),\Vect)\\
&\xrightarrow{\sim} \PolyFun(\Sh_G(S),\Vect)=\Hop_G(S)
\end{align*}

This equivalence obviously commutes with the sheaf operations on the $S$ component, and hence induces an equivalence of SSH structures between $S\mapsto \PPS{(S\times \Irr G)}$ (as in \autoref{prop:shonpower}) and $\Hop_G$.

In particular we have that $\PPS{\Irr G}$ is equivalent to $\PP_G$ as a symmetric monoidal category, since the symmetric monoidal structure on each of them comes from the SSH structure that they are a part of.

\begin{Remark}
\label{rem:EquiToPpower}
The functor $Y$ has an inverse, which we denote $Y^{-1}$ given by
\[
Y^{-1}(V)(s)=\bigoplus_{\rho\in\Irr G}V_{(s,\rho)}\otimes \rho
\]
This will be of use to us in the next subsection.

Note that this functor is defined up to a choice of concrete models for the irreducible representations of $G$.
\end{Remark}

\subsection{Connection with wreath products}
In \cite{Zelbook}, Zelevinsky considers the PSH algebra $R_S[G]:=\bigoplus_n K\left( \Rep(S_n[G])\right)$ where $S_n[G]$ is the wreath product $G^n\rtimes S_n$, multiplication is given by induction and inner product is given by dimension of hom-space.

\begin{Proposition}
There is a natural map of PSH algebras $R_S[G]\rightarrow K(\PP_G)$, and this map is an isomorphism.
\end{Proposition}

\begin{proof}
We will first construct a contravariant exact functor

\[
L:\bigoplus_n \Rep(S_n[G])\rightarrow\PP_G
\]
Denote $G_n:=S_n[G]$, and let $\rho$ be a representation of $G_n$. We define $L_\rho\in\PP_G$ as follows: Take $V\in\Sh_G(\fset{1})=\Rep(G)$. Then $V^{\otimes n}$ is naturally a representation of $G_n$, and we define \[
L_\rho(V):=\Hom_{G_n}(\rho,V^{\otimes n})
\]

Since we are in characteristic 0, this functor is obviously exact, so induces a map of $K$-groups, and we need to check that it preserves all the relevant structures.

First, it preserves multiplication because 
\begin{align*}
m(L_{\rho_1},L_{\rho_2})(V)&=L_{\rho_1}(V)\otimes L_{\rho_2}(V)\\
&= \Hom_{G_{n_1}}(\rho_1,V^{\otimes n_1})\otimes\Hom_{G_{n_2}}(\rho_2,V^{\otimes n_2})\\
&\cong \Hom_{G_{n_1}\times G_{n_2}}(\rho_1\boxtimes \rho_2,V^{\otimes n_1}\boxtimes V^{\otimes n_2})\\
&\cong \Hom_{G_{n_1+n_2}}(\rho_1\cdot\rho_2,V^{\otimes (n_1+n_2)})\\
&=L_{\rho_1\cdot\rho_2}(V)
\end{align*}

Secondly, it is fully faithful, since by evaluating on $\kk[G]$ we see that

$\Hom(L_{\rho_1},L_{\rho_2})$ is canonically isomorphic to $\Hom({\rho_2},{\rho_1})$ with the isomorphism given by the map $\Hom({\rho_2},{\rho_1})\rightarrow\Hom(L_{\rho_1},L_{\rho_2})$ induced by $L$. The upshot is that the map of algebras preserves the inner product, and in all is a map of PSH algebras.

Now, compose the functor $L$ with the equivalence induced by $Y^{-1}$ (from \autoref{rem:EquiToPpower}). Since both functors preserve multiplication, we get a map of algebras $R_S[G]\rightarrow K(\PPS{\Irr G})$.

If $\rho\in\Irr G$, this composition of functors takes $\rho$ to a polynomial functor 
\[
R_\rho:\Sh(\Irr G)\rightarrow \Vect
\]
given by\[
R_\rho(V)=V_\rho
\]
It is obvious that $R_\rho$ defines a primitive irreducible element of $K(\PPS{\Irr G})$. Moreover, these are all of the irreducible primitives in this algebra.

Also, as noted in \cite{Zelbook}, $\Irr G$ is the set of all irreducible primitives in $R_S[G]$. In all, we have a map of PSH algebras, which sends the set of irreducible primitives on the LHS bijectively onto the set of irreducible primitives on the RHS, so it must be an isomorphism.

As a consequence, the map induced by $L$ is also an isomorphism.

\end{proof}


\subsection{Generalization}
In the above, the role played by the group $G$ is all via the algebra $\kk[G]$. We can therefore replace $\kk[G]$ by some algebra $A$ and repeat the same construction. This is especially interesting for $A$ which is not semisimple, and this will be addressed in a future work. For instance the case of $A=\CC[x]/x^2$ is related to a categorification of the quantum group $U_q(sl_2)$ (see \cite{CautisLicataQuantum}).

\appendix
\section{Commutative cubes}
\label{app:commcube}
In this section we explain what we mean by the term "commutative cube" in a bicategory. The general statement is as follows: for an $n$-dimensional cube given an ordering of the coordinates there is a concise way to orient the 2-morphisms, thus getting a diagram of 2-morphisms for any cube. We will say that the cube is commutative if this diagram is. The ordering used in this article is described below. A good general reference for this topic is \cite{graycoherence}. 

\subsection{2-cubes}

Consider a square, where we have ordered the coordinates as $x<y$. Then we orient the 2-morphism by the lexicographical order, i.e. $xy\rightarrow yx$.

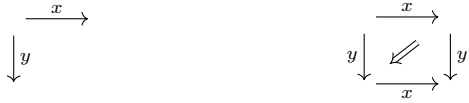
\begin{align*}
\stik{1}{
{} \ar{r}{x} \ar{d}{y} \& {}\\
{} \& {}
}
&
{}
&
\stik{1}{
{} \ar{r}{x} \ar{d}[left]{y}\& {} \ar{d}{y}\ar[Rightarrow,shorten <=1em,shorten >=1em]{dl}[above,sloped]{}\\
{} \ar{r}[below]{x}\& {}
}
\end{align*}

\subsection{3-cubes}
Order the three coordinates in the cube as $x<y<z$. The edges of the cube are all oriented positively in one of these directions, i.e. as in the diagram:
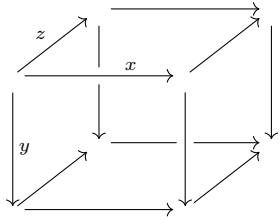
\[\stik{1}{
{} \& {} \ar{rr} \ar{dd} \& {} \& {} \ar{dd}\\
{} \ar[crossing over]{rr}[xshift=1pc]{x} \ar{dd}{y} \ar{ur}{z} \& {} \& {} \ar{dd}{} \ar{ur} \& {}\\
{} \& {} \ar{rr} \& {} \& {}\\
{} \ar{rr} \ar{ur} \& {} \& {} \ar{ur}
\latearrow{commutative diagrams/crossing over}{2-3}{4-3}{}
}
\]
The 2-morphisms are oriented as in the previous section. E.g. we have a morphism $xy\rightarrow yx$, and via whiskering we get a morphism $xyz\rightarrow  yxz$.

In this way we get a diagram of all 2-morphisms between full paths on the cube, which is a hexagon
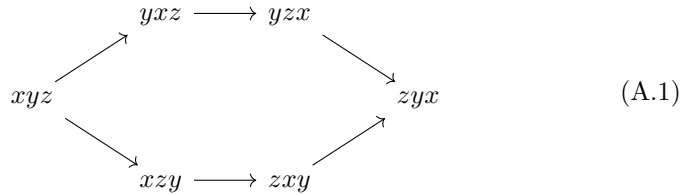
\begin{equation}
\label{app:2morphoncube}
\stik{1}{
\& yxz \ar{r} \& yzx \ar{dr}\\
xyz \ar{ur} \ar{dr} \& \& \& zyx\\
\& xzy \ar{r}\& zxy \ar{ur}
}
\end{equation}

We say that the cube is commutative if this diagram commutes.

We can draw this diagram in the cube in the following way:
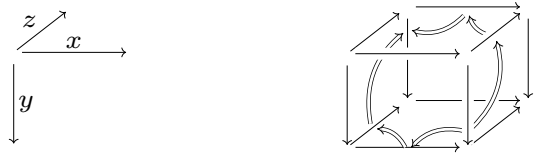
\begin{align*}
\stik{0.7}{
{} \& {} \& {}\\
{} \ar{ur}{\scalebox{2}{$z$}} \ar{rr}{\scalebox{2}{$x$}} \ar{dd}{\scalebox{2}{$y$}} \& {} \& {}\\
{} \& {} \& {} \\
{} \& {} \& {}
}
&
{}
&
\stik{0.7}{
{} \& {} \ar{rr}[name=topback,below]{} \ar{dd}[name=backleft,sloped,xshift=-1.5pc]{} \& {} \& {} \ar{dd}\\
{} \ar[crossing over]{rr}[name=topfront,above]{}\ar{dd}[name=frontleft]{} \ar{ur}[name=topleft,below]{} \& {} \& {} \ar{dd}[name=frontright,sloped,xshift=1pc]{} \ar{ur}[name=topright,below]{} \& {}\\
{} \& {} \ar{rr} \& {} \& {}\\
{} \ar{rr}[name=bottomfront,below]{} \ar{ur}[name=bottomleft,sloped,yshift=-0.2pc]{} \& {} \& {} \ar{ur}
\arrow[Rightarrow,to path={(topright) to[bend left] (topback)}]{}{}
\arrow[Rightarrow,to path={(frontright) to[bend right] (bottomfront)}]{}{}
\arrow[Rightarrow,crossing over,to path={(frontright) to[bend right] (topright)}]{}{}
\arrow[Rightarrow,to path={(bottomfront) to[bend right] (bottomleft)}]{}{}
\arrow[Rightarrow,to path={(bottomleft) to[bend left] (backleft)}]{}{}
\arrow[Rightarrow,to path={(topback) to[bend left] (backleft)}]{}{}
\latearrow{/tikz/commutative diagrams/crossing over}{2-3}{4-3}{}
\latearrow{/tikz/commutative diagrams/crossing over}{2-1}{2-3}{}
\latearrow{/tikz/commutative diagrams/crossing over}{2-3}{1-4}{}
}
\end{align*}

\begin{Lemma}
\label{lem:cubefront}
\label{app:cubecommutes}
Consider a cube in a bicategory:
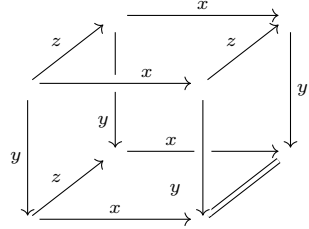
\[\stik{1}{
{} \& {} \ar{rr}{x} \ar{dd}[left,yshift=-1pc]{y} \& {} \& {} \ar{dd}{y}\\
{} \ar[crossing over]{rr}[xshift=1pc]{x} \ar{dd}[left]{y} \ar{ur}{z} \& {} \& {} \ar{dd}[yshift=-1pc,left,xshift=-0.5pc]{y} \ar{ur}{z} \& {}\\
{} \& {} \ar{rr}[xshift=-1pc]{x} \& {} \& {}\\
{} \ar{rr}{x} \ar{ur}{z} \& {} \& {} \ar[equal]{ur}
\latearrow{commutative diagrams/crossing over}{2-3}{4-3}{}
}
\]
with all faces given except the front, and such that the left and bottom faces are invertible, then there is a unique 2-morphism that fits in the front face and makes the cube commute.
\end{Lemma}
\begin{proof}
The graph of 2-morphisms is
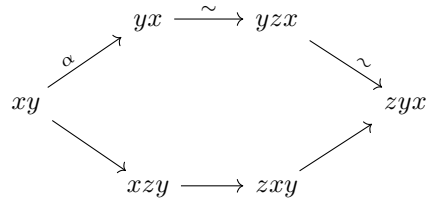
\[
\stik{1}{
\& yx \ar{r}[sloped]{\sim} \& yzx \ar{dr}[sloped]{\sim}\\
xy \ar{ur}[sloped]{\alpha} \ar{dr} \& \& \& zyx\\
\& xzy \ar{r}\& zxy \ar{ur}
}
\]
so the morphism $\alpha:xy\rightarrow yx$ that makes this diagram commute exists and is unique.
\end{proof}

\subsection{Higher dimensional cubes}
In the same way we can orient the 2-morphisms on any $k$-cube by ordering the coordinates as $x_1,\ldots,x_k$.

For instance for a 4-cube the resulting diagram is a 3-dimensional shape with 8 faces which are hexagons and 6 faces which are squares (a \emph{truncated octahedron}). The entire diagram commutes iff each face commutes. The hexagons correspond to the sub 3-cubes, and the squares are related to the four-interchange law in a bicategory. So we see that a 4-cube in a bicategory commutes iff every sub 3-cube in it commutes.

More generally we have
\begin{Theorem}[Gray, \cite{graycoherence}]
A $k$-cube in a bicategory commutes iff every sub 3-cube in it commutes.
\end{Theorem}

\section{The Beck-Chevalley condition for squares and cubes}
\label{sec:BC}
\label{app:BeckChevalley}
Let $\CCC$ be a bicategory, and consider a square (sometimes called a \emph{quintet})
\begin{equation}
\label{bcsquarebefore}
\tik
A\arrow[]{d}{f} \arrow[]{r}{g}& B \arrow[]{d}{h} \arrow[shorten >=0.4cm,shorten <=0.4cm,Rightarrow]{dl}[above,sloped]{\alpha} \\ 
C \arrow[]{r}{i} & D
\tak
\end{equation}
i.e. a 2-morphism $\alpha:h\circ g\rightarrow i\circ f$ (which is not necessarily an isomorphism).

Suppose that the verticals $h,f$ both have right adjoints $h_R,f_R$, with given unit-counit pairs, then we can form the square (called the \emph{right mate} of the above square)
\[
\tik
A \arrow[]{r}{g}\arrow[blue,shorten >=0.4cm,shorten <=0.4cm,Rightarrow]{dr}[above,sloped]{\alpha_R}& B   \\ 
C \arrow[dashed,blue,thick]{u}{f_R} \arrow[]{r}{i} & D \arrow[dashed,blue,thick]{u}[right]{h_R}
\tak
\]
where $\alpha_R$ is the composition
\[
g\circ f_R\rightarrow h_R\circ h \circ g \circ f_R \xrightarrow{\alpha} h_R \circ i \circ f \circ f_R \rightarrow h_R \circ i
\]
\begin{Definition}
A square as in \eqref{bcsquarebefore} with 2-morphism $\alpha$ is said to satisfy the right \emph{Beck-Chevalley condition} if $\alpha_R$ is invertible.
\end{Definition}
Similarly, if the horizontals $g,i$ both have left adjoints, we can define the left Beck-Chevalley condition via the left mate square
\[
\tik
A\arrow[]{d}{f} & B\arrow[dashed,red,thick]{l}[above]{g_L} \arrow[]{d}{h}  \\ 
C  & D\arrow[dashed,red,thick]{l}[below]{i_L} \arrow[red,shorten >=0.4cm,shorten <=0.4cm,Rightarrow]{ul}[above,sloped]{\alpha_L}
\tak
\]
\begin{Remark}
Note that a square satisfies the right BC condition iff it satisfies the left BC condition, when both are defined, hence we can omit the words \emph{left} or \emph{right}.
\end{Remark}

To get relations between mates of squares we need to look at cubes.

\begin{Lemma}
\label{lem:cubemate}
Consider a commutative cube (see \autoref{app:commcube} for details)
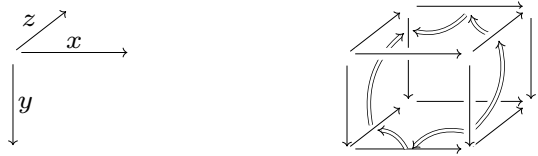
\begin{align*}
\stik{0.7}{
{} \& {} \& {}\\
{} \ar{ur}{\scalebox{2}{$z$}} \ar{rr}{\scalebox{2}{$x$}} \ar{dd}{\scalebox{2}{$y$}} \& {} \& {}\\
{} \& {} \& {} \\
{} \& {} \& {}
}
&
{}
&
\stik{0.7}{
{} \& {} \ar{rr}[name=topback,below]{} \ar{dd}[name=backleft,sloped,xshift=-1.5pc]{} \& {} \& {} \ar{dd}\\
{} \ar[crossing over]{rr}[name=topfront,above]{}\ar{dd}[name=frontleft]{} \ar{ur}[name=topleft,below]{} \& {} \& {} \ar{dd}[name=frontright,sloped,xshift=1pc]{} \ar{ur}[name=topright,below]{} \& {}\\
{} \& {} \ar{rr} \& {} \& {}\\
{} \ar{rr}[name=bottomfront,below]{} \ar{ur}[name=bottomleft,sloped,yshift=-0.2pc]{} \& {} \& {} \ar{ur}
\arrow[Rightarrow,to path={(topright) to[bend left] (topback)}]{}{}
\arrow[Rightarrow,to path={(frontright) to[bend right] (bottomfront)}]{}{}
\arrow[Rightarrow,crossing over,to path={(frontright) to[bend right] (topright)}]{}{}
\arrow[Rightarrow,to path={(bottomfront) to[bend right] (bottomleft)}]{}{}
\arrow[Rightarrow,to path={(bottomleft) to[bend left] (backleft)}]{}{}
\arrow[Rightarrow,to path={(topback) to[bend left] (backleft)}]{}{}
\latearrow{/tikz/commutative diagrams/crossing over}{2-3}{4-3}{}
\latearrow{/tikz/commutative diagrams/crossing over}{2-1}{2-3}{}
\latearrow{/tikz/commutative diagrams/crossing over}{2-3}{1-4}{}
}
\end{align*}
and suppose that all arrows in the $x$ (resp. $z$) direction have left (resp. right) adjoints, with given unit/counit. Then the left (right) mate of the cube, i.e. the cube obtained from taking the left (right) mates of all faces involving direction $x$ ($z$) arrows, commutes. Explicitly, they are the following cubes
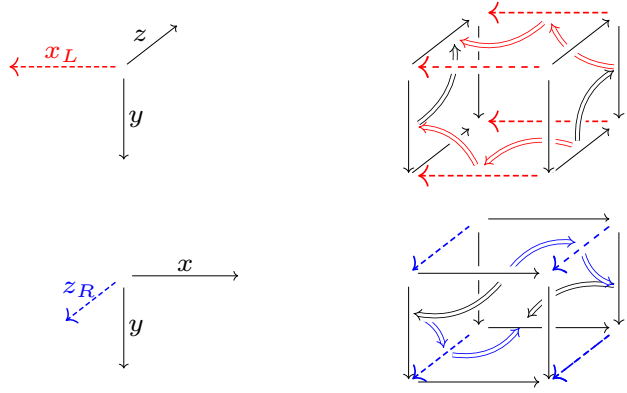
\begin{align*}
\stik{0.7}{
{} \& {} \& {} \& {} \& {}\\
{} \& {} \& {} \ar{ur}{\scalebox{2}{$z$}} \ar[dashed,red,thick]{ll}[above]{\scalebox{2}{$x_L$}} \ar{dd}{\scalebox{2}{$y$}} \& {} \& {}\\
{} \& {} \& {} \& {} \& {} \\
{} \& {} \& {} \& {} \& {}
}
&
{}
&
\stik{0.8}{
{} \& {} \ar[leftarrow,dashed,red,thick]{rr}[name=topback,below]{} \ar{dd}[name=backleft,sloped,xshift=-1.5pc]{} \& {} \& {} \ar{dd}[name=backright,sloped,left]{}\\
{} \ar[leftarrow,dashed,red,thick,crossing over]{rr}[name=topfront,above]{}\ar{dd}[name=frontleft]{} \ar{ur}[name=topleft,below,sloped,xshift=0.5pc]{} \& {} \& {} \ar{dd}[name=frontright,sloped,xshift=1pc]{} \ar{ur}[name=topright,below]{} \& {}\\
{} \& {} \ar[leftarrow,dashed,red,thick]{rr}[name=bottomback,above]{} \& {} \& {}\\
{} \ar[leftarrow,dashed,red,thick]{rr}[name=bottomfront,above]{} \ar{ur}[name=bottomleft,sloped,yshift=-0.2pc]{} \& {} \& {} \ar{ur}[name=bottomright,sloped,above]{}
\arrow[Rightarrow,red,to path={(bottomright) to[bend right] (bottomfront)}]{}{}
\arrow[Rightarrow,red,crossing over,to path={(bottomfront) to[bend right] (frontleft)}]{}{}
\arrow[Rightarrow,to path={(frontleft) to[bend right] (topleft)}]{}{}
\arrow[Rightarrow,crossing over,to path={(bottomright) to[bend left] (backright)}]{}{}
\arrow[Rightarrow,red,to path={(backright) to[bend left] (topback)}]{}{}
\arrow[Rightarrow,red,crossing over,to path={(topback) to[bend left] (topleft)}]{}{}
\latearrow{/tikz/commutative diagrams/crossing over}{2-3}{4-3}{}
\latearrow{/tikz/commutative diagrams/leftarrow,dashed,red,thick,/tikz/commutative diagrams/crossing over}{2-1}{2-3}{}
\latearrow{/tikz/commutative diagrams/crossing over}{2-3}{1-4}{}
} \\
\stik{0.7}{
{} \& {} \& {} \& {} \& {}\\
{} \& {} \& {} \ar[dashed,thick,blue]{dl}[above,xshift=-0.5pc]{\scalebox{2}{$z_R$}} \ar{rr}[above]{\scalebox{2}{$x$}} \ar{dd}{\scalebox{2}{$y$}} \& {} \& {}\\
{} \& {} \& {} \& {} \& {} \\
{} \& {} \& {} \& {} \& {}
}
&
{}
&
\stik{0.8}{
{} \& {} \ar{rr}[name=topback,below]{} \ar{dd}[name=backleft,sloped,xshift=-1.5pc]{} \& {} \& {} \ar{dd}[name=backright,sloped,left,xshift=0.5pc]{}\\
{} \ar[crossing over]{rr}[name=topfront,below,xshift=1pc]{} \ar{dd}[name=frontleft,sloped,xshift=-1pc]{} \ar[leftarrow,dashed,blue,thick]{ur}[name=topleft,above,sloped,xshift=0.5pc]{} \& {} \& {} \ar{dd}[name=frontright,sloped,xshift=1pc]{} \ar[leftarrow,dashed,blue,thick]{ur}[name=topright,above]{} \& {}\\
{} \& {} \ar{rr}[name=bottomback,above,xshift=-1pc]{} \& {} \& {}\\
{} \ar{rr}[name=bottomfront,above]{} \ar[leftarrow,dashed,blue,thick]{ur}[name=bottomleft,sloped,right]{} \& {} \& {} \ar[leftarrow,dashed,blue,thick]{ur}[name=bottomright,sloped,below]{}
\arrow[Rightarrow,to path={(backright) to[bend right] (bottomback)}]{}{}
\arrow[Rightarrow,blue,to path={(frontleft) to[bend left] (bottomleft)}]{}{}
\arrow[Rightarrow,blue,to path={(bottomleft) to[bend right] (bottomback)}]{}{}
\arrow[Rightarrow,blue,to path={(topfront) to[bend left] (topright)}]{}{}
\arrow[Rightarrow,blue,to path={(topright) to[bend right] (backright)}]{}{}
\arrow[Rightarrow,crossing over,to path={(topfront) to[bend left] (frontleft)}]{}{}
\latearrow{/tikz/commutative diagrams/crossing over}{2-3}{4-3}{}
\latearrow{/tikz/commutative diagrams/crossing over}{2-1}{2-3}{}
\latearrow{/tikz/commutative diagrams/leftarrow,dashed,blue,thick,/tikz/commutative diagrams/crossing over}{2-3}{1-4}{}
\latearrow{/tikz/commutative diagrams/leftarrow,dashed,blue,thick}{4-3}{3-4}{}
}
\end{align*}
Using the description from Appendix $\ref{app:commcube}$ the first cube corresponds to the ordering of coordinates $y<z<x$ and the second to the ordering $z<x<y$.
\end{Lemma}
\begin{proof}
In a 2-category this is a direct computation, using the 4-interchange law, and the fact that the original cube commutes. 
\end{proof}

\section{Adjunction of inverse and direct image for sheaves on finite sets}
\label{app:pullpushadjsets}
Let $\varphi:S\rightarrow T$ be a map of sets, and let $\varphi^*,\varphi_*$ be the functors of inverse and direct image between the categories of sheaves. Let $V\in\Sh(S),W\in\Sh(T),A\subset S,B\subset T$. Note that for a sheaf on a finite set we have \[
V(A)=\prod_{a\in A}V_a
\]
so we have the following formulas:
\begin{gather}
\varphi^*W(A)=\prod_{a\in A}W_{\varphi(a)}\\
\varphi_*V(B)=V(\varphi^{-1}(B))
\end{gather}
and so
\begin{gather}
\varphi^*\varphi_*V(A)=\prod_{a\in A}V(\varphi^{-1}(\varphi(a)))\\
\varphi_*\varphi^*W(B)=\prod_{a,\varphi(a)\in B}W_{\varphi(a)}
\end{gather}
The unit and counit of the adjunction $\varphi^*\dashv\varphi_*$ are given by maps
\begin{gather}
\epsilon_R: \varphi^*\varphi_*V(A)=\prod_{a\in A}V(\varphi^{-1}(\varphi(a))) \rightarrow \prod_{a\in A}V_a=V(A)\\
\eta_R: W(B)=\prod_{b\in B}W_b\rightarrow \prod_{a,\varphi(a)\in B}W_{\varphi(a)}=\varphi_*\varphi^*W(B)
\end{gather}
where $\epsilon_R$ is given by restrictions and $\eta_R$ by the diagonal maps.

There is also an adjunction $\varphi_*\dashv\varphi^*$ given by maps
\begin{gather}
\eta_L: \varphi^*\varphi_*V(A)=\prod_{a\in A}V(\varphi^{-1}(\varphi(a))) \leftarrow \prod_{a\in A}V_a=V(A)\\
\epsilon_L: W(B)=\prod_{b\in B}W_b\leftarrow \prod_{a,\varphi(a)\in B}W_{\varphi(a)}=\varphi_*\varphi^*W(B)
\end{gather}
where $\eta_L$ is given by extension by 0, and $\epsilon_L$ is given by the sum maps.


\section{Limits}
\label{app:PBcont}

\subsection{Weak adjoints}
\begin{Definition}[see \cite{Gray2adj}]
Let $F:\CCC\rightarrow\DDD$, $G:\DDD\rightarrow\CCC$ be functors of bi-categories, then we say that $G$ is \emph{weak right adjoint} to $G$ if there is a natural functor
\[
\Hom(FC,D)\leftarrow\Hom(C,GD)
\]
which admits a right adjoint.
\end{Definition}

We will mostly be interested in the case where $\DDD$ is trivial, where this amounts to saying that we have some chosen object $G(\point)$, and for any $C$ the map $\Hom(C,G(\point))\rightarrow \point$ has a right adjoint, which is the same as saying that $\Hom(C,G(\point))$ has a final object. 
\subsection{Conical limits}
Let $K$ be a small 1-category and $\CCC$ a 1-category.

Let $F:K\rightarrow \CCC$ be a functor. The limit of $F$ can be thought of as follows (cf. \cite{grandisparedouble}):

Denote also by $F$ the functor $\point\rightarrow\CCC^K$ which sends $\point$ to $F$, and by $diag$ the diagonal embedding $\CCC\rightarrow\CCC^K$. Then we may form the comma category
\[
\stik{1}{
(diag\downarrow F) \ar{d} \ar{r}[above]{p} \& \CCC   \ar{d}{diag}\\
\point \ar{r}{F} \&  \CCC^K
}
\]
Note that a final object $L$ in $(diag\downarrow F)$ is exactly a limit of $F$ together with all the relevant morphisms to the diagram. If we want just the object, then it is given by $p(L)$. So the condition of having a limit is the same as the requirement that $(diag\downarrow F)$ has a final object, which is the same as requiring that the map  $(diag\downarrow F)\rightarrow \point$ has a right adjoint.

\subsection{Weak limits and \comma squares}
\label{app:WeightedLimits}
In the above, we only dealt with conical limits, i.e. we used the map $\CCC\xrightarrow{diag}\CCC^K$ which comes from the map $K\rightarrow\point$.

However, we can replace it by any map $K\xrightarrow{W} L$, and form the square
\[
\stik{1}{
(W^*\downarrow F) \ar{d} \ar{r}[above]{p} \& \CCC   \ar{d}{W^*}\\
\point \ar{r}{F} \&  \CCC^K
}
\]
We now say that $F$ admits a $W$-limit if the map $(W\downarrow F)\rightarrow \point$ admits a weak right adjoint (this is closely related to the notion of \emph{weighted} limit appearing in the literature).

\begin{Example}[\comma Squares]
Let $K\rightarrow L$ be the map 
\[
\stik{1}{
{} \&  \bullet\ar{d} \\
 \bullet\ar{r} \& \bullet
}
\xhookrightarrow{W}
\stik{1}{
 \bullet \ar{r} \ar{d} \&  \bullet\ar{d} \ar[Rightarrow,shorten <=0.7em,shorten >=0.7em]{dl}\\
 \bullet\ar{r} \& \bullet
}
\]
and let $\CCC$ be a bicategory. 

A map $K\xrightarrow{F}\CCC$ is a "bottom right corner", and we call a $W$-limit of $F$ an \comma square with this bottom right corner (if $\CCC$ is a 1-category this conicides with the usual notion of Cartesian square).
\end{Example}

Let us analyze this in more detail:

Let $F=\stik{1}{
{} \&  A\ar{d}{f} \\
B\ar{r}{g} \& C
}$ be a bottom right corner in $\CCC$, and consider the square:
\[
\stik{1}{
(W^*\downarrow F) \ar{d} \ar{r}[above]{p} \& \CCC^\square  \ar{d}{W^*}\\
\point \ar{r}{F} \&  \CCC^\lrcorner
}
\]

Objects of $(W^*\downarrow F)$ are pairs $(S\in \CCC^\square,\alpha:W^*S\rightarrow F)$. So they can be thought of as diagrams:
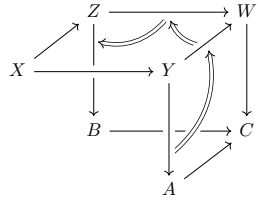
\[
\stik{0.7}{
{} \& Z \ar{rr}[name=topback,below]{} \ar{dd}[name=backleft,sloped,xshift=-1.5pc]{} \& {} \& W \ar{dd}\\
X \ar[crossing over]{rr}[name=topfront,above]{} \ar{ur}[name=topleft,below]{} \& {} \& Y \ar{dd}[name=frontright,sloped,xshift=1pc]{} \ar{ur}[name=topright,below]{} \& {}\\
{} \& B \ar{rr} \& {} \& C\\
 \& {} \& A \ar{ur}
\arrow[Rightarrow,to path={(topright) to[bend left] (topback)}]{}{}
\arrow[Rightarrow,crossing over,to path={(frontright) to[bend right] (topright)}]{}{}
\arrow[Rightarrow,to path={(topback) to[bend left] (backleft)}]{}{}
\latearrow{/tikz/commutative diagrams/crossing over}{2-3}{4-3}{}
\latearrow{/tikz/commutative diagrams/crossing over}{2-1}{2-3}{}
\latearrow{/tikz/commutative diagrams/crossing over}{2-3}{1-4}{}
}
\]
A morphism in  $(W^*\downarrow F)$ is then a commutative diagram
\[
\stik{0.7}{
{} \& Z' \ar{rr} \ar{dd} \& {} \& W' \ar{dd}\\
X' \ar[crossing over]{rr} \ar{dd} \ar{ur} \& {} \& Y' \ar{dd} \ar{ur} \& {}\\
{} \& Z \ar{rr}[name=topback,below]{} \ar{dd}[name=backleft,sloped,xshift=-1.5pc]{} \& {} \& W \ar{dd}\\
X \ar[crossing over]{rr}[name=topfront,above]{} \ar{ur}[name=topleft,below]{} \& {} \& Y \ar{dd}[name=frontright,sloped,xshift=1pc]{} \ar{ur}[name=topright,below]{} \& {}\\
{} \& B \ar{rr} \& {} \& C\\
 \& {} \& A \ar{ur}
}
\]
together with a map
\[
\stik{0.7}{
{} \& Z' \ar{rr} \ar{dd} \& {} \& W' \ar{dd}\\
 \& {} \& Y' \ar{dd} \ar{ur} \& {}\\
{} \& Z \ar{rr}[name=topback,below]{} \ar{dd}[name=backleft,sloped,xshift=-1.5pc]{} \& {} \& W \ar{dd}\\
 \& {} \& Y \ar{dd}[name=frontright,sloped,xshift=1pc]{} \ar{ur}[name=topright,below]{} \& {}\\
{} \& B \ar{rr} \& {} \& C\\
 \& {} \& A \ar{ur}
} \rightarrow \stik{1}{
{} \& Z' \ar{rr}[name=topback,below]{} \ar{dd}[name=backleft,sloped,xshift=-1.5pc]{} \& {} \& W' \ar{dd}\\
 \& {} \& Y' \ar{dd}[name=frontright,sloped,xshift=1pc]{} \ar{ur}[name=topright,below]{} \& {}\\
{} \& B \ar{rr} \& {} \& C\\
 \& {} \& A \ar{ur}
}
\]

\end{document}